\documentclass[reqno,11pt]{amsart}
\usepackage{amscd,amssymb,verbatim,array}
\usepackage{hyperref}

\numberwithin{equation}{section} \theoremstyle{plain}

\newcommand{\Complex}{\mathbb C}
\newcommand{\Real}{\mathbb R}

\newcommand{\ddbar}{\overline\partial}
\newcommand{\pr}{\partial}
\newcommand{\ol}{\overline}
\newcommand{\Td}{\widetilde}
\newcommand{\norm}[1]{\left\Vert#1\right\Vert}
\newcommand{\abs}[1]{\left\vert#1\right\vert}
\newcommand{\set}[1]{\left\{#1\right\}}
\newcommand{\To}{\rightarrow}

\newtheorem{theorem}{Theorem}[section]
\newtheorem{definition}[theorem]{Definition}
\newtheorem{proposition}[theorem]{Proposition}
\newtheorem{corollary}[theorem]{Corollary}
\newtheorem{lemma}[theorem]{Lemma}
\newtheorem{remark}[theorem]{Remark}

\begin{document}

\title[Equidistribution theorems on strongly pseudoconvex domains]
{Equidistribution theorems on strongly pseudoconvex domains}

\author{Chin-Yu Hsiao}

\address{Institute of Mathematics, Academia Sinica and National Center for Theoretical Sciences, Astronomy-Mathematics Building, No. 1, Sec. 4, Roosevelt Road, Taipei 10617, Taiwan}
\thanks{The first author was partially supported by Taiwan Ministry of Science and Technology project 104-2628-M-001-003-MY2 , the Golden-Jade fellowship of Kenda Foundation and Academia Sinica Career Development Award. This work was initiated when the second author was visiting the Institute of Mathematics at Academia Sinica in the summer of 2016. The second author would like to thank the Institute of Mathematics at Academia Sinica for its hospitality and financial support during his stay. The second author was also supported by Taiwan Ministry of Science and Technology project 105-2115-M-008-008-MY2}

\email{chsiao@math.sinica.edu.tw or chinyu.hsiao@gmail.com}

\author{Guokuan Shao}

\address{Institute of Mathematics, Academia Sinica, Astronomy-Mathematics Building, No. 1, Sec. 4, Roosevelt Road, Taipei 10617, Taiwan}

\email{guokuan@gate.sinica.edu.tw}
	
\begin{abstract}
This work consists of two parts. In the first part, we consider a compact connected strongly pseudoconvex CR manifold $X$ with a transversal CR $S^{1}$ action. We establish an equidistribution theorem on zeros of CR functions. The main techniques involve a uniform estimate of Szeg\H{o} kernel on $X$. 

In the second part, we consider a general complex manifold $M$ with a strongly pseudoconvex boundary $X$.  By using classical result of Boutet de Monvel-Sj\"ostrand about Bergman kernel asymptotics,  we establish an equidistribution theorem on zeros of holomorphic functions on $\ol M$. 
\end{abstract}
	
	\maketitle \tableofcontents
	

\section{Introduction and statement of the main results }\label{s-gue170727}

The study of equidistribution of zeros of holomorphic sections has become intensively active in recent years. Shiffman-Zelditch \cite{sz} established an equidistribution property for high powers of a positive line bundle. Dinh-Sibony \cite{ds} extended the equidistribution with estimate of convergence speed and applied to general measures. More results about equidistribution of zeros of holomorphic sections in different cases, such as line bundles with singular metrics, general base spaces, general measures, were obtained in \cite{cm,cmm,cmn,dmm,dms,sh1,sh2}. Important methods to study equidistribution include uniform estimates for Bergman kernel functions \cite{mm, t} and techniques for complex dynamics in higher dimensions \cite{fs}. Our article is the first to study equidistribution on CR manifolds and on complex manifolds with boundary. 
In the first part, we establish an equidistribution theorem on zeros of CR functions. The proof involves uniform estimates for Szeg\H{o} kernel functions \cite{hhl}.
In the second part, we consider a general complex manifold $M$ with a strongly pseudoconvex boundary $X$ and we establish an equidistribution theorem on zeros of holomorphic functions on $\ol M$ by using classical result of Boutet de Monvel-Sj\"ostrand~\cite{BouSj76}.  

We now state our main results. We refer to Section~\ref{s:prelim} for some notations and terminology used here. 
Let $(X, T^{1,0}X)$ be a compact connected strongly pseudoconvex CR manifold with a transversal CR $S^{1}$ action $e^{i\theta}$ (cf. Section 2), where $T^{1,0}X$ is a CR structure of $X$. The dimension of $X$ is $2n+1$, $n\geq 1$. Denote by $T\in C^{\infty}(X, TX)$ the real vector field induced by the $S^{1}$ action. Take a $S^1$ invariant  Hermitian metric $\langle\,\cdot\,|\,\cdot\,\rangle$ on $\mathbb{C}TX$ such that there is an orthogonal decomposition
$\mathbb{C}TX=T^{1,0}X\oplus T^{0,1}X\oplus\mathbb{C}T$.
Then there exists a natural global $L^{2}$ inner product $(\,\cdot\,|\,\cdot\,)$ on $C^{\infty}(X)$ induced by $\langle\cdot|\cdot\rangle$. 

For every $q\in\mathbb N$, put
\begin{equation*}
X_q:=\{x\in X: e^{i\theta}\circ x\neq x, \forall\theta\in(0,\frac{2\pi}{q}),\ e^{i\frac{2\pi}{q}}\circ x=x\}.
\end{equation*}
Set $p:=\min\{q\in\mathbb{N}: X_{q}\neq\emptyset \}$. Put $X_{{\rm reg\,}}=X_p$. For simplicity, we assume that $p=1$. Since $X$ is connected, $X_{1}$ is open and dense in $X$. Assume that 
$X=\cup_{j=0}^{t-1}X_{p_{j}}$, $1=p_0<p_1<\cdots<p_{t-1}$ and put $X_{\rm sing}:=\cup_{j=1}^{t-1}X_{p_{j}}$. 

Let $\bar\partial_{b}: C^\infty(X)\To\Omega^{0,1}(X)$ be the tangential Cauchy-Riemann operator.
For each $m\in\mathbb{Z}$, put
\begin{equation}
H_{b,m}^{0}(X):=\{u\in C^{\infty}(X):Tu=imu, \bar\partial_{b}u=0  \}.
\end{equation}
It is well-known that $\dim H_{b,m}^{0}(X)<\infty$ (see~\cite{hl}). 
Let $f_1\in H^0_{b,m}(X),\ldots,f_{d_m}\in H^0_{b,m}(X)$ be an orthonormal basis for $H^0_{b,m}(X)$. The Szeg\H{o} kernel function associated to $H^0_{b,m}(X)$ is given by 
\[S_m(x):=\sum^{d_m}_{j=1}\abs{f_j(x)}^2.\]
When the $S^1$ action is globally free, it is well-known that $S_m(x)\approx m^n$ uniformly on $X$. When $X$ is locally free, we only have $S_m(x)\approx m^n$ locally uniformly on $X_{{\rm reg\,}}$ in general (see Theorem~\ref{t-gue170704}). Moreover, $S_m(x)$ can be zero at some point of $X_{\rm sing}$ even for $m$ large (see~\cite{hl} and~\cite{hhla}). Let 
\begin{equation}\label{e-gue170704rya}
\alpha=[p_1,\ldots,p_{t-1}],
\end{equation}
that is $\alpha$ is the least common multiple of $p_1,\ldots,p_{t-1}$. 
In  Theorem~\ref{t-gue170704r}, we will show that there exist positive integers $1=k_0<k_{1}<\cdot\cdot\cdot<k_{t-1}$ independent of $m$ such that 
\[cm^n\leq S_{\alpha m}(x)+S_{k_1\alpha m}(x)+\cdots+S_{k_{t-1}\alpha m}(x)\leq\frac{1}{c}m^n\ \ \mbox{on $X$}\]
for all $m\gg1$, where $0<c<1$ is a constant independent of $m$. 
For each $m\in\mathbb N$, put 
\begin{equation}\label{e-gue170794lr}
 A_m(X):=\bigcup^{t-1}_{j=0}H^0_{b,k_j\alpha m}(X).
 \end{equation}
We write $d\mu_m$ to denote the normalized Haar measure on the unit sphere, 
defined in the natural way by using a fixed orthonormal basis,
\[SA_m(X):=\set{g\in A_m(X);\, (\,g\,|\,g\,)=1}.\]
Let $a_m={\rm dim\,}A_m(X)$. We fix an orthonormal basis $\set{g^{(m)}_j}^{a_m}_{j=1}$ of $A_m(X)$ with respect to $(\,\cdot\,|\,\cdot\,)$, then we can identify the sphere $S^{2a_m-1}$ to $SA_m(X)$ by 
\[(z_1,\ldots,z_{a_m})\in S^{2a_m-1}\To\sum^k_{j=1}z_jg^{(m)}_j\in SA_m(X),\]
and we have 
\begin{equation}\label{e-gue170703a}
d\mu_m=\frac{dS^{2a_m-1}}{{\rm vol\,}(S^{2a_m-1})},
\end{equation}
where $dS^{2a_m-1}$ denotes the standard Haar measure on $S^{2a_m-1}$. We consider the probability space $\Omega(X):=\prod^\infty_{m=1}SA_m(X)$ with the probability measure $d\mu:=\prod^\infty_{m=1}d\mu_m$. We denote $u=\set{u_m}\in\Omega(X)$.  

Since the $S^1$ action is transversal and CR, 
$X\times \mathbb{R}$ is a complex manifold with the following holomorphic tangent bundle and complex structure $J$,
\begin{equation}\label{e-gue180419}
\begin{split}
&T^{1,0}X\oplus\{\mathbb{C}(T-i\frac{\partial}{\partial\eta}) \},\\
&JT=\frac{\partial}{\partial\eta}, \ Ju=iu \ \text{for}\ u\in T^{1,0}X.\\
\end{split}
\end{equation}
Let $v(z,\theta,\eta)$ be a non-trivial holomorphic function on $X\times\Real$. We write $[v=0]$ to denote the current of integration with multiplicities over the analytic hypersurface $\{v=0\}$ determined by the nontrivial holomorphic function $v$ on $X\times\Real$.  That is, for a smooth $2n$-form $g\in\Omega_{0}^{2n}(X\times\Real)$  with compact support in $X\times\Real$, we have 
\begin{equation}\label{e-reI}
\langle[v=0], g\rangle=\int_{\{v=0\}}g.
\end{equation}
Denote by $\tilde\partial$ (resp. $\bar{\tilde\partial}$) the $\partial$-operator (resp. $\ddbar$-operator)
with respect to the complex structure in \eqref{e-gue180419}. By the Lelong-Poincar\'{e} formula~\cite[III-2.15]{dj} and \cite[Theorem 2.3.3]{mm} (see Propositioin~\ref{p-gue1804}), we have
\begin{equation}\label{e-gue1804191}
\langle[v=0], g\rangle=\frac{i}{2\pi}\int\tilde\partial\bar{\tilde\partial}\log|v|^{2}\wedge g.
\end{equation}

For $u\in A_m(X)$, it is easy to see that  there exists a unique function $v(x,\eta)\in C^\infty(X\times\Real)$, which is holomorphic in $X\times \mathbb{R}$ such that $v\big|_{\eta=0}=u$ (see Lemma~\ref{l-gue170704w}). For all  $g\in\Omega^{p,q}(X)$, we extend $g$ trivially in the variable $\eta$ on $X\times\Real$. Then, for $f\in\Omega^{n,n}_0(X)$ and every $\chi(\eta)\in C^\infty_0(\Real)$, $f\wedge\omega_0\wedge\chi(\eta)d\eta$ is a smooth $2n$-from on $X\times\Real$ with compact support in $X\times\Real$. We then define $\langle[v=0], f\wedge\omega_0\wedge\chi(\eta)d\eta\rangle$ as in \eqref{e-reI}. 
The main result of the first part is the following 

\begin{theorem}\label{t-gue170704ryz}
With the above notations and assumptions,  fix $\chi(\eta)\in C^\infty_0(\Real)$ with $\int\chi(\eta)d\eta=1$ and let $\varepsilon_m$ be a sequence with $\lim_{m\To\infty}m\varepsilon_m=0$. Then for $d\mu$-almost every $u=\set{u_m}\in\Omega(X)$, we have 
\begin{equation}\label{e-reII}
\lim_{m\To\infty}\frac{1}{m}\langle\,[v_{m}=0], f\wedge\omega_0\wedge \frac{1}{\varepsilon_m}\chi(\frac{\eta}{\varepsilon_m})d\eta \,\rangle
=\alpha\frac{1+k_{1}^{n+1}+\cdot\cdot\cdot+k_{t-1}^{n+1}}{1+k_{1}^{n}+\cdot\cdot\cdot+k_{t-1}^{n}}\frac{i}{\pi}\int_{X}\mathcal{L}_{X}\wedge f\wedge\omega_{0},
\end{equation}
for all $f\in\Omega^{n-1,n-1}(X)$, where $v_m(x,\eta)$ is the unique holomorphic function on $X\times \mathbb{R}$ such that $v_m\big|_{\eta=0}=u_m(x)$, $\alpha=[p_1,\ldots,p_{t-1}]$, $f\wedge\omega_0\wedge \frac{1}{\varepsilon_m}\chi(\frac{\eta}{\varepsilon_m})d\eta$ is a smooth $(n,n)$ form on $X\times\Real$, the duality $\langle\,\cdot\,,\cdot\,\rangle$ in \eqref{e-reII} is given by \eqref{e-reI}, $\eta$ denotes the coordinate on $\Real$, $\omega_0$ is the Reeb one form on $X$ (see the discussion in the beginning of Section~\ref{s-gue170627}),
$\mathcal{L}_{X}$ denotes the Levi form of $X$ with respect to the Reeb one form $\omega_0$ (see Definition~\ref{d-1.2}). 
\end{theorem} 

\begin{remark}\label{r-gue180419}
We explain the role $\varepsilon_m$ in Theorem~\ref{t-gue170704ryz}.
For simplicity, assume that $t=2$ and $m_1:=\alpha m$, $m_2:=\alpha k_1m$. 
Let $u\in H^0_{b,m_1}(X)\oplus H^0_{b,m_2}(X)$. Let $(z,\theta,\varphi)$ be BRT coordinates on an open set $D$ of $X$ (see Theorem~\ref{t-gue170720y}).
On $D$, we write 
\begin{equation*}
u=u_{1}+u_{2}=\tilde u_{1}(z)e^{im_{1}\theta}+\tilde u_{2}(z)e^{im_{2}\theta}\in H^{0}_{b,m_{1}}(X)\oplus H^{0}_{b,m_{2}}(X).
\end{equation*}
Then the unique holomorphic function $v(z,\theta,\eta)\in C^\infty(X\times\Real)$ with $v\big|_{\eta=0}=u$ is given by 
\[v(z,\theta,\eta)=\tilde u_{1}(z)e^{im_{1}(\theta+i\eta)}+\tilde u_{2}(z)e^{im_{2}(\theta+i\eta)}.\]
Then, formally
\begin{equation}\label{e-gue180420}\begin{split}
&\langle\,[v(z,\theta,\eta)=0], f(z,\theta)\wedge\omega_0\wedge\frac{1}{\varepsilon_m}\chi(\frac{\eta}{\varepsilon_m})d\eta \,\rangle\\
&=\langle\,[\tilde u_{1}(z)e^{im_{1}(\theta+i\eta)}+\tilde u_{2}(z)e^{im_{2}(\theta+i\eta)}=0],  f(z,\theta)\wedge\omega_0\wedge\frac{1}{\varepsilon_m}\chi(\frac{\eta}
{\varepsilon_m})d\eta \,\rangle\\
&=\langle\,[\tilde u_{1}(z)e^{im_{1}(\theta+i\varepsilon_m\eta)}+\tilde u_{2}(z)e^{im_{2}(\theta+i\varepsilon_m\eta)}=0],  f(z,\theta)\wedge\omega_0\wedge\chi(\eta)d\eta \,\rangle.
\end{split}\end{equation}
From the last equation of \eqref{e-gue180420}, intuitively speaking, when $m\varepsilon_m\To0$, the integral 
\[\langle\,[v(z,\theta,\eta)=0], f(z,\theta)\wedge\omega_0\wedge\frac{1}{\varepsilon_m}\chi(\frac{\eta}{\varepsilon_m})d\eta \,\rangle\] 
will converge to the  integration of ``CR" current $\langle\,[u(z,\theta)=0], f(z,\theta)\wedge\omega_0\,\rangle$.
\end{remark}

\begin{remark}\label{r-gue180420}
Assume that the $S^1$-action is globally free. Let $u\in H^0_{b,m}(X)$. Let $(z,\theta,\varphi)$ be BRT coordinates on an open set $D$ of $X$ (see Theorem~\ref{t-gue170720y}).
On $D$, we write $u=\tilde u(z)e^{im\theta}$ and the unique holomorphic function $v(z,\theta,\eta)\in C^\infty(X\times\Real)$ with $v\big|_{\eta=0}=u$ is given by 
$v(z,\theta,\eta)=u(z)e^{im(\theta+i\eta)}$. Then $\{v=0\}=\{u=0\}\times\Real$ and for every $\varepsilon_m>0$, we have 
\[\begin{split}\langle\,[v(z,\theta,\eta)=0], f(z,\theta)\wedge\omega_0\wedge\frac{1}{\varepsilon_m}\chi(\frac{\eta}{\varepsilon_m})d\eta \,\rangle
=\langle\,[v(z,\theta,\eta)=0], f(z,\theta)\wedge\omega_0\wedge\chi(\eta)d\eta \,\rangle.\end{split}\]
For the globally free case, we don't need $\varepsilon_m$ in Theorem~\ref{t-gue170704ryz}.
\end{remark}
When the $S^1$-action is globally free, then $t=1$, $\alpha=1$, $A_{m}(X)=H^{0}_{b,m}(X)$
and $\Omega(X)=\prod_{m=1}^{\infty}SA_{m}(X)=\prod_{m=1}^{\infty}SH^{0}_{b,m}(X)$. 
From Remark~\ref{r-gue180420}, we deduce the following

\begin{corollary}\label{c-gue180418}
With the same notations and assumptions in Theorem \ref{t-gue170704ryz}, if the $S^1$-action is globally free, then for $d\mu$-almost every $u=\set{u_m}\in\Omega(X)$ we have 
\begin{equation}\label{e-reII1}
\lim_{m\To\infty}\frac{1}{m}\langle\,[v_{m}=0], f\wedge\omega_0\wedge\chi(\eta)d\eta \,\rangle
=\frac{i}{\pi}\int_{X}\mathcal{L}_{X}\wedge f\wedge\omega_{0},
\end{equation}
for all $f\in\Omega^{n-1,n-1}(X)$, where $v_m$ is the unique holomorphic function on $X\times \mathbb{R}$ such that $v_m\big|_{\eta=0}=u_m$. 
\end{corollary}

Let $Y$ be a compact K\"ahler manifold with 
$\dim_{\Complex}Y=n$ and let $L$ be a line bundle over $Y$ with a smooth Hermitian metric $h$ such that the induced curvature $R^L$ is positive on $Y$. 
Let $e_{L}$ be a local frame of $L$. We write $|e_{L}(y)|_{h}=e^{-\phi}$.  Then $R^{L}=2\partial\ddbar\phi$. Take $\omega:=\frac{i}{2\pi}R^{L}$ to be the K\"ahler form of $Y$. Denote by $H^{0}(Y,L^{m})$ the space of all holomorphic sections of $L^m$.
Set $\Omega(Y,L):=\prod_{m=1}^{\infty}SH^{0}(Y,L^{m})$ with the probability measure
$d\mu:=\prod^\infty_{m=1}d\mu_m$ (cf. \eqref{e-gue170703a})
As an application of Theorem \ref{t-gue170704ryz}, we obtain the classical equidistribution theorem on line bundles (see e.g. \cite[Theorem 1.1]{sz} and \cite[Theorem 5.3.3]{mm}).

\begin{corollary}\label{c-gue1804181}
With the above notations and assumptions, for $d\mu$-almost every $s=\{s_{m}\}\in\Omega(Y,L)$,
we have
\begin{equation}\label{e-reII2}
\lim_{m\To\infty}\frac{1}{m}[s_{m}=0]=\omega
\end{equation}
in the sense of currents.
\end{corollary}

Now we formulate the main result of the second part. 
Let $M$ be a relatively compact open subset with $C^\infty$ boundary $X$ of a complex manifold $M'$ of dimension $n+1$ with a smooth Hermitian metric $\langle\,\cdot\,|\,\cdot\,\rangle$ on its holomorphic tangent bundle $T^{1,0}M'$. The Hermitian metric on holomorphic tangent bundle induces a Hermitian metric $\langle\,\cdot\,|\,\cdot\,\rangle$ on $\oplus^{2n+2}_{k=1}\Lambda^k(\Complex T^*M')$. 
Let $r\in C^\infty(M',\Real)$ be a defining function of $X$, that is, $X=\set{z\in M';\, r(z)=0}$, $M=\set{z\in M';\, r(z)<0}$. We take $r$ so that $\norm{dr}^2=\langle\,dr\,|\,dr\,\rangle=1$ on $X$. In this work, we assume that $X$ is strongly pseudoconvex, that is, $\pr\ddbar r|_{T^{1,0}X}$ is positive definite at each point of $X$, where $T^{1,0}X:=T^{1,0}M'\bigcap\Complex TX$ is the standard CR structure on $X$. Let $dv_M$ be the volume form on $M$ induced by $\langle\,\cdot\,|\,\cdot\,\rangle$ and let $(\,\cdot\,|\,\cdot\,)_M$ be the $L^2$ inner product on $C^\infty_0(M)$ induced by $dv_M$ and let $L^2(M)$ be the completion of $C^\infty_0(M)$ with respect to $(\,\cdot\,|\,\cdot\,)_M$. Let $H_{(2)}^0(M)=\set{u\in L^2(M);\, \ddbar u=0}$. By using classical result of Boutet de Monvel-Sj\"ostrand~\cite[Theorem 1.5]{BouSj76}, we see that $C^\infty(\ol M)\bigcap H_{(2)}^0(M)$ is dense in $H_{(2)}^0(M)$ in the $L^2(M)$ space and we can find $g_j\in C^\infty(\ol M)\bigcap H_{(2)}^0(M)$ with $(\, g_j\,|\,g_k\,)_M=\delta_{j,k}$, $j, k=1,2,\ldots$, such that the set
\begin{equation}\label{e-gue170709}
A(M):=\rm{span\,}\set{g_1, g_2,\ldots}
\end{equation}
is dense in $H_{(2)}^0(M)$. That is, for every $h\in H_{(2)}^0(M)$, we can find $h_{\ell}\in A(M)$, $\ell=1,2,\ldots$, such that $\lim_{\ell\To\infty}h_{\ell}=h$ in $L^2(M)$ space. 

To state our equidistribution theorem, we need to introduce some notations. 
For every $m\in\mathbb N$, let $A_m(M)={\rm span\,}\set{g_1,\ldots,g_m}$, where $g_j\in H_{(2)}^0(M)\bigcap C^\infty(\ol M)$, $j=1,\ldots,m$, are as \eqref{e-gue170709}. Let $d\mu_m$ be the equidistribution probability measure on the unit sphere 
\[SA_m(M):=\set{g\in A_m(M);\, (\,g\,|\,g\,)_M=1}.\]
Let $\beta:=\set{b_j}^\infty_{j=1}$ with $b_1<b_2<\cdots$ and $b_j\in\mathbb N$, for every $j=1,2,\ldots$. 
We consider the probability space 
\begin{equation}\label{e-gue170709I}
\Omega(M,\beta):=\prod^\infty_{j=1}SA_{b_j}(M)
\end{equation} 
with the probability measure 
\begin{equation}\label{e-gue170709II}
d\mu(\beta):=\prod^\infty_{j=1}d\mu_{b_j}.
\end{equation}
We denote $u=\set{u_k}\in\Omega(M,\beta)$. For $g\in H_{(2)}^0(M)\bigcap C^\infty(\ol M)$, we let $[g=0]$ denote the zero current in $M$.

Let $B^{*0,1}M'=\set{u\in T^{*0,1}M';\, \langle\,u\,|\,\ddbar r\,\rangle=0}$, where $T^{*0,1}M'$ denotes the bundle of $(0,1)$ forms on $M'$. Let $B^{*1,0}M':=\ol{B^{*0,1}M'}$ and let $B^{*p,q}M':=\Lambda^p(B^{*1,0}M')\wedge\Lambda^q(B^{*0,1}M')$, $p,q=1,\ldots,n$. Let $\omega_0=J(dr)$, where $J$ is the standard complex structure map on $T^*M'$ and let $\mathcal{L}_X\in C^\infty(X,T^{*1,1}X)$ be the Levi form induced by $\omega_0$ (see Definition~\ref{d-1.2}). Our second main result is the following

\begin{theorem}\label{t-gue170703c}
With the notations and assumptions above, fix $\psi\in C^\infty_0([-1,-\frac{1}{2}])$. There exists a sequence $\beta=\set{b_j}^\infty_{j=1}$ independent of $\psi$ with $b_1<b_2<\cdots$, $b_j\in\mathbb N$, $j=1,2,\ldots$, such that for $d\mu(\beta)$-almost every $u=\set{u_k}\in\Omega(M,\beta)$, we have 
\begin{equation}\label{e-gue170703c}
\lim_{k\To\infty}\langle\,[u_k=0], (2i)kr\psi(kr)\phi\wedge\pr r\wedge\ddbar r\,\rangle=-(n+2)\frac{i}{2\pi}c_0\int_X\mathcal{L}_X\wedge\omega_0\wedge\phi
\end{equation}
for all $\phi\in C^\infty(\ol M,B^{*n-1,n-1}M')$, where $c_0=\int_\Real\psi(x)dx$, $\Omega(M,\beta)$ and $d\mu(\beta)$ are as in \eqref{e-gue170709I} and \eqref{e-gue170709II} respectively. 
\end{theorem}

\begin{remark}\label{r-gue180417}
By the result of Boutet de Monvel-Sj\"ostrand~\cite[Theorem 1.5]{BouSj76}, we have 
\[\mbox{$\sum^\infty_{j=1}\abs{g_j(x)}^2\sim\abs{r^{-(n+2)}(x)}$ in $\ol M$}.\]
The numbers $b_{j}$	are chosen so that the function  $\sum^{b_j}_{s=1}\abs{g_s(x)}^2\sim\abs{r^{-(n+2)}(x)}$ on $\{x\in M:-\frac{1}{k}\leq r\leq-\frac{1}{2k}\}$ (see Theorem \ref{t-gue170717} ). In general, 
$\sum^{j}_{s=1}\abs{g_s(x)}^2$ could not be asymptotically $\abs{r^{-(n+2)}(x)}$ and we can't not take $b_j$ to be $j$. It is an interesting question to determine the subsequence $b_j$. 
\end{remark}

\begin{remark}\label{r-gue180420a}
Note that for any smooth $(n,n)$ form on $M$, we can write $\phi\wedge\pr r\wedge\ddbar r$ near the boundary $X$,  where $\phi\in C^\infty(\ol M,B^{*n-1,n-1}M')$. 
From the proof of Theorem~\ref{t-gue170703c}, we actually prove that for $d\mu(\beta)$-almost every $u=\set{u_k}\in\Omega(M,\beta)$, we have 
\begin{equation}\label{e-gue180420q}
\lim_{k\To\infty}\Bigr(\frac{1}{k}\langle\,[u_k=0], g_k\,\rangle+i\frac{n+2}{2\pi}\frac{1}{k}\int_{M}\pr\ddbar\log(-r)\wedge g_k\Bigr)=0,
\end{equation}
for all $k$-uniformly test form $g_k\in\Omega^{n,n}_0(M)$. Here $k$-uniformly test form $g_k\in\Omega^{n,n}_0(M)$ means that for any smooth $(1,1)$ form $\psi$, the integral 
$\int g_k\wedge\psi$ is uniformly bounded in $k$. For example, $k^2r\psi(kr)\phi\wedge\pr r\wedge\ddbar r$ is a $k$-uniformly test form, where $\psi\in C^\infty_0([-1,-\frac{1}{2}])$ and $\phi\in C^\infty(\ol M,B^{*n-1,n-1}M')$. In Theorem~\ref{t-gue170703c}, we take special test form $r\psi(kr)\phi\wedge\pr r\wedge\ddbar r$ since 
Theorem~\ref{t-gue170703c} aims to show the asymptotic behavior of the currents $\{[u_{k}=0]\}$ when the supports of test forms tend to approach the boundary $X$.
\end{remark}

The paper is organized as follows. In Section~\ref{s:prelim} we collect some notations we use throughout and we recall the basic knowledge about CR manifolds. In Section 3 we recall a theorem about Szeg\H{o} kernel asymptotics and give a uniform estimate of Szeg\H{o} kernel functions. 
Section 4 is devoted to proving Theorem~\ref{t-gue170704ryz}. In Section~\ref{s-gue170709}, we first construct holomorphic functions with specific rate near the boundary and we prove Theorem~\ref{t-gue170703c}.  

\section{Preliminaries}\label{s:prelim}
\subsection{Standard notations} \label{s-ssna}
We shall use the following notations: $\mathbb N=\set{1,2,\ldots}$, $\mathbb N_0=\mathbb N\cup\set{0}$, $\Real$ 
is the set of real numbers, $\ol\Real_+:=\set{x\in\Real;\, x\geq0}$. 
For a multi-index $\alpha=(\alpha_1,\ldots,\alpha_n)\in\mathbb N_0^n$,
we denote by $\abs{\alpha}=\alpha_1+\ldots+\alpha_n$ its norm and by $l(\alpha)=n$ its length.
For $m\in\mathbb N$, write $\alpha\in\set{1,\ldots,m}^n$ if $\alpha_j\in\set{1,\ldots,m}$, 
$j=1,\ldots,n$. $\alpha$ is strictly increasing if $\alpha_1<\alpha_2<\ldots<\alpha_n$. For $x=(x_1,\ldots,x_n)$, we write
\[
\begin{split}
&x^\alpha=x_1^{\alpha_1}\ldots x^{\alpha_n}_n,\\
& \pr_{x_j}=\frac{\pr}{\pr x_j}\,,\quad
\pr^\alpha_x=\pr^{\alpha_1}_{x_1}\ldots\pr^{\alpha_n}_{x_n}=\frac{\pr^{\abs{\alpha}}}{\pr x^\alpha}\,,\\
&D_{x_j}=\frac{1}{i}\pr_{x_j}\,,\quad D^\alpha_x=D^{\alpha_1}_{x_1}\ldots D^{\alpha_n}_{x_n}\,,
\quad D_x=\frac{1}{i}\pr_x\,.
\end{split}
\]
Let $z=(z_1,\ldots,z_n)$, $z_j=x_{2j-1}+ix_{2j}$, $j=1,\ldots,n$, be coordinates of $\Complex^n$.
We write
\[
\begin{split}
&z^\alpha=z_1^{\alpha_1}\ldots z^{\alpha_n}_n\,,\quad\ol z^\alpha=\ol z_1^{\alpha_1}\ldots\ol z^{\alpha_n}_n\,,\\
&\pr_{z_j}=\frac{\pr}{\pr z_j}=
\frac{1}{2}\Big(\frac{\pr}{\pr x_{2j-1}}-i\frac{\pr}{\pr x_{2j}}\Big)\,,\quad\pr_{\ol z_j}=
\frac{\pr}{\pr\ol z_j}=\frac{1}{2}\Big(\frac{\pr}{\pr x_{2j-1}}+i\frac{\pr}{\pr x_{2j}}\Big),\\
&\pr^\alpha_z=\pr^{\alpha_1}_{z_1}\ldots\pr^{\alpha_n}_{z_n}=\frac{\pr^{\abs{\alpha}}}{\pr z^\alpha}\,,\quad
\pr^\alpha_{\ol z}=\pr^{\alpha_1}_{\ol z_1}\ldots\pr^{\alpha_n}_{\ol z_n}=
\frac{\pr^{\abs{\alpha}}}{\pr\ol z^\alpha}\,.
\end{split}
\]
For $j, s\in\mathbb Z$, set $\delta_{j,s}=1$ if $j=s$, $\delta_{j,s}=0$ if $j\neq s$.

Let $W$ be a $C^\infty$ paracompact manifold.
We let $TW$ and $T^*W$ denote the tangent bundle of $W$
and the cotangent bundle of $W$, respectively.
The complexified tangent bundle of $W$ and the complexified cotangent bundle of $W$ will be denoted by $\Complex TW$
and $\Complex T^*W$, respectively. Write $\langle\,\cdot\,,\cdot\,\rangle$ to denote the pointwise
duality between $TW$ and $T^*W$.
We extend $\langle\,\cdot\,,\cdot\,\rangle$ bilinearly to $\Complex TW\times\Complex T^*W$.
Let $G$ be a $C^\infty$ vector bundle over $W$. The fiber of $G$ at $x\in W$ will be denoted by $G_x$.
Let $E$ be a vector bundle over a $C^\infty$ paracompact manifold $W_1$. We write
$G\boxtimes E^*$ to denote the vector bundle over $W\times W_1$ with fiber over $(x, y)\in W\times W_1$
consisting of the linear maps from $E_y$ to $G_x$.  Let $Y\subset W$ be an open set. 
From now on, the spaces of distribution sections of $G$ over $Y$ and
smooth sections of $G$ over $Y$ will be denoted by $D'(Y, G)$ and $C^\infty(Y, G)$, respectively.
Let $E'(Y, G)$ be the subspace of $D'(Y, G)$ whose elements have compact support in $Y$. Put $C^\infty_0(Y,G):=C^\infty(Y,G)\bigcap E'(Y,G)$. 

Let $G$ and $E$ be $C^\infty$ vector
bundles over paracompact orientable $C^\infty$ manifolds $W$ and $W_1$, respectively, equipped with smooth densities of integration. If
$A: C^\infty_0(W_1,E)\To D'(W,G)$
is continuous, we write $K_A(x, y)$ or $A(x, y)$ to denote the distribution kernel of $A$.

Let $H(x,y)\in D'(W\times W_1,G\boxtimes E^*)$. We write $H$ to denote the unique 
continuous operator $C^\infty_0(W_1,E)\To D'(W,G)$ with distribution kernel $H(x,y)$. 
In this work, we identify $H$ with $H(x,y)$. 

Let $M$ be a relatively compact open subset with $C^\infty$ boundary $X$ of a complex manifold $M'$. Let $F$ be a $C^\infty$ vector bundle over $M'$. Let $C^\infty(\ol M, F)$, $D'(\ol M,F)$ denote the spaces of restrictions to $M$ of elements in the spaces $C^\infty(M',F)$, $D'(M',F)$ respectively. 

\subsection{CR manifolds}\label{s-gue170627}

Let $(X, T^{1,0}X)$ be a compact, orientable CR manifold of dimension $2n+1$, $n\geq 1$, where $T^{1,0}X$ is a CR structure of $X$, that is, $T^{1,0}X$ is a subbundle of rank $n$ of the complexified tangent bundle $\mathbb{C}TX$, satisfying $T^{1,0}X\cap T^{0,1}X=\{0\}$, where $T^{0,1}X=\overline{T^{1,0}X}$, and $[\mathcal V,\mathcal V]\subset\mathcal V$, where $\mathcal V=C^\infty(X, T^{1,0}X)$. 
We  fix a real non-vanishing $1$ form $\omega_0\in C(X,T^*X)$ so that $\langle\,\omega_0(x)\,,\,u\,\rangle=0$, for every $u\in T^{1,0}_xX\oplus T^{0,1}_xX$, for every $x\in X$.
We call $\omega_0$ Reeb one form on $X$. 

\begin{definition}\label{d-1.2}
For $p\in X$, the Levi form $\mathcal L_{X,p}$ of $X$ at $p$ is the Hermitian quadratic form on $T^{1,0}_pX$ given by
$\mathcal{L}_{X,p}(U,V)=-\frac{1}{2i}\langle\,d\omega_0(p)\,,\,U\wedge\ol V\,\rangle$, $U, V\in T^{1,0}_pX$.
	
	Denote by $\mathcal{L}_{X}$ the Levi form on $X$.
\end{definition} 

Fix a global non-vanishing vector field $T\in C^\infty(X,TX)$ such that $\omega_0(T)=-1$ and $T$ is transversal to $T^{1,0}X\oplus T^{0,1}X$. We call $T$ Reeb vector field on $X$. 
Take a smooth Hermitian metric $\langle \cdot \mid \cdot \rangle$ on $\mathbb{C}TX$ so that $T^{1,0}X$ is orthogonal to $T^{0,1}X$, $\langle u \mid v \rangle$ is real if $u, v$ are real tangent vectors, $\langle\,T\,|\,T\,\rangle=1$ and $T$ is orthogonal to $T^{1,0}X\oplus T^{0,1}X$. For $u \in \mathbb{C}TX$, we write $|u|^2 := \langle u | u \rangle$. Denote by $T^{*1,0}X$ and $T^{*0,1}X$ the dual bundles $T^{1,0}X$ and $T^{0,1}X$, respectively. They can be identified with subbundles of the complexified cotangent bundle $\mathbb{C}T^*X$. Define the vector bundle of $(p,q)$-forms by $T^{*p,q}X := (\wedge^pT^{*1,0}X)\wedge(\wedge^qT^{*0,1}X)$. The Hermitian metric $\langle \cdot | \cdot \rangle$ on $\mathbb{C}TX$ induces, by duality, a Hermitian metric on $\mathbb{C}T^*X$ and also on the bundles of $(p,q)$ forms $T^{*p,q}X, p, q=0, 1, \cdots, n$. We shall also denote all these induced metrics by $\langle \cdot | \cdot \rangle$. Note that we have the pointwise orthogonal decompositions:
\begin{equation}
\begin{array}{c}
\mathbb{C}T^*X = T^{*1,0}X \oplus T^{*0,1}X \oplus \left\{ \lambda \omega_0: \lambda \in \mathbb{C} \right\}, \\
\mathbb{C}TX = T^{1,0}X \oplus T^{0,1}X \oplus \left\{ \lambda T: \lambda \in \mathbb{C} \right\}.
\end{array}
\end{equation} 

Let $D$ be an open set of $X$. Let $\Omega^{p,q}(D)$ denote the space of smooth sections of $T^{*p,q}X$ over $D$ and let $\Omega^{p,q}_0(D)$ be the subspace of $\Omega^{p,q}(D)$ whose elements have compact support in $D$. For each point $x\in X$, in this paper, we will identify $\mathcal L_{X,x}$ as a $(1,1)$ form at $x$. Hence, 
$\mathcal L_X\in\Omega^{1,1}(X)$.  


Now, we assume that $X$ admits an $S^1$-action: $S^1\times X\rightarrow X, (e^{i\theta}, x)\rightarrow e^{i\theta}\circ x$. Here we use $e^{i\theta}$ to denote the $S^1$-action. Let $\Td T\in C^\infty(X, TX)$ be the global real vector field induced by the $S^1$-action given as follows
\begin{equation}
(\Td Tu)(x)=\frac{\partial}{\partial\theta}\left(u(e^{i\theta}\circ x)\right)\Big|_{\theta=0},~u\in C^\infty(X).
\end{equation}
\begin{definition}
	We say that the $S^1$-action $e^{i\theta} ~(0\leq\theta<2\pi$) is CR if
	$$[\Td T, C^\infty(X, T^{1,0}X)]\subset C^\infty(X, T^{1,0}X),$$
	where $[~,~ ]$ is the Lie bracket between the smooth vector fields on $X$.
	Furthermore, the $S^1$-action is called transversal if for each $x\in X$ one has
	$$\mathbb C\Td T(x)\oplus T_x^{1,0}(X)\oplus T_x^{0,1}X=\mathbb CT_xX.$$
\end{definition}

If the $S^1$ action is transversal and CR, we will always  take the Reeb one form on $X$ to be the global real one form determined by 
$\langle\,\omega_0\,,\, u\,\rangle=0$, for every $u\in T^{1,0}X\oplus T^{0,1}X$ and $\langle\,\omega_0\,,\,\Td T\,\rangle=-1$ and we will  always  take the Reeb vector field on $X$ to be $\Td T$. Hence, we will also write $T$ to denote the global real vector field induced by the $S^1$-action. 

Until further notice, we assume that $(X, T^{1,0}X)$ is a compact connected strongly pseudoconvex CR manifold with a transversal CR $S^1$-action $e^{i\theta}$.  For every $q\in\mathbb N$, put
\begin{equation}
X_q:=\{x\in X: e^{i\theta}\circ x\neq x, \forall\theta\in(0,\frac{2\pi}{q}),\ e^{i\frac{2\pi}{q}}\circ x=x\}.
\end{equation}
Set $p:=\min\{q\in\mathbb{N}: X_{q}\neq\emptyset \}$.
Thus, $X_{{\rm reg\,}}=X_p$. Note that one can re-normalize the $S^1$-action by lifting such that the new $S^1$-action satisfies  $X_1\neq\emptyset$, see \cite{dh}. For simplicity, we assume that $p=1$. If $X$ is connected, then $X_{1}$ is open and dense in $X$. Assume that \begin{equation*}
X=\cup_{j=0}^{t-1}X_{p_{j}},\ \ 1=:p_0<p_1<\cdots<p_{t-1}. 
\end{equation*}
Put $X_{sing}:=X_{sing}^{1}=\cup_{j=1}^{t-1}X_{p_{j}}$, and $X_{sing}^{r}:=\cup_{j=r}^{t-1}X_{p_{j}}$ for $2\leq r\leq t-1$.
Take the convention that $X_{sing}^{t}=\emptyset$.
It follows from \cite{dh} that

\begin{proposition}
$X_{sing}^{r}$ is a closed subset of $X$, for $1\leq r\leq t$.
\end{proposition}

Fix $\theta_0\in [0, 2\pi)$. Let
$$d e^{i\theta_0}: \mathbb CT_x X\rightarrow \mathbb CT_{e^{i\theta_0}x}X$$
denote the differential map of $e^{i\theta_0}: X\rightarrow X$. By the properties of transversal CR $S^1$-actions, we can check that
\begin{equation}\label{a}
\begin{split}
de^{i\theta_0}:T_x^{1,0}X\rightarrow T^{1,0}_{e^{i\theta_0}x}X,\\
de^{i\theta_0}:T_x^{0,1}X\rightarrow T^{0,1}_{e^{i\theta_0}x}X,\\
de^{i\theta_0}(T(x))=T(e^{i\theta_0}x).
\end{split}
\end{equation}
Let $(e^{i\theta_0})^\ast: \Lambda^q(\mathbb CT^\ast X)\rightarrow\Lambda^q(\mathbb CT^\ast X)$ be the pull back of $e^{i\theta_0}, q=0,1\cdots, 2n+1$. From \eqref{a}, we can check that for every $q=0, 1,\cdots, n$
\begin{equation*}
(e^{i\theta_0})^\ast:  T^{\ast0,q}_{e^{i\theta_0}x}X\rightarrow T_x^{\ast0,q}X.
\end{equation*}

Let $u\in\Omega^{0,q}(X)$. The Lie derivative of $u$ along the direction $T$ is denoted by $Tu$.
We have $Tu\in\Omega^{0, q}(X)$ for all $u\in\Omega^{0, q}(X)$.

Let $\overline\partial_b:\Omega^{0,q}(X)\rightarrow\Omega^{0,q+1}(X)$ be the tangential Cauchy-Riemann operator. From \eqref{a}, it is straightforward to check that
\begin{equation}\label{c}
T\overline\partial_b=\overline\partial_bT~\text{on}
~\Omega^{0,q}(X).
\end{equation}

For every $m\in\mathbb Z$, put $\Omega^{0,q}_m(X):=\{u\in\Omega^{0,q}(X): Tu=imu\}$. For $q=0$, we write $C^\infty_m(X):=\Omega^{0,0}_m(X)$. 
We denote by $\overline\partial_{b, m}$ the restriction of $\overline\partial_b$ to $\Omega^{0, q}_m(X)$.
From (\ref{c}) we have the $\ddbar_{b, m}$-complex for every $m\in\mathbb Z$:
\begin{equation*}\label{e-gue140903VI}
\ddbar_{b, m}:\cdots\To\Omega^{0,q-1}_m(X)\To\Omega^{0,q}_m(X)\To\Omega^{0,q+1}_m(X)\To\cdots.
\end{equation*}
For $m\in\mathbb Z$, the $q$-th $\ddbar_{b, m}$-cohomology is given by
\begin{equation}\label{a8}
H^{q}_{b,m}(X):=\frac{{\rm Ker\,}\ddbar_{b}:\Omega^{0,q}_m(X)\To\Omega^{0,q+1}_m(X)}{\operatorname{Im}\ddbar_{b}:\Omega^{0,q-1}_m(X)\To\Omega^{0,q}_m(X)}.
\end{equation}
Moreover,  we have \cite{hl}
\begin{equation*}\label{a1}
{\rm dim} H^q_{b, m}(X)<\infty, ~\text{for all}~ q=0, \ldots, n.
\end{equation*}

\begin{definition}\label{def-16-09-01}
	A function $u\in C^\infty(X)$ is a Cauchy-Riemann function (CR function for short)
	if $\overline\partial_bu=0$, that is $\overline Zu=0$ for all $Z\in C^{\infty}(X, T^{1, 0}X)$. For $m\in \mathbb N$, $H^0_{b, m}(X)$ is called the $m$-th positive Fourier component of the space of CR functions.
\end{definition}
We recall the canonical local coordinates (BRT coordinates) due to Baouendi-Rothschild-Treves, (see \cite{brt}).

\begin{theorem}\label{t-gue170720y}
	With the notations and assumptions above, fix $x_0\in X$. There exist local coordinates $(x_1,\cdots,x_{2n+1})=(z,\theta)=(z_1,\cdots,z_{n},\theta), z_j=x_{2j-1}+ix_{2j}, 1\leq j\leq n, x_{2n+1}=\theta$, centered at $x_0$, defined on 
	$D=\{(z, \theta)\in\mathbb C^{n}\times\mathbb R: |z|<\varepsilon, |\theta|<\delta\}$, such that 
	\begin{equation}\label{e-can}
	\begin{split}
	&T=\frac{\partial}{\partial\theta}\\
	&Z_j=\frac{\partial}{\partial z_j}+i\frac{\partial\varphi(z)}{\partial z_j}\frac{\partial}{\partial\theta},j=1,\cdots,n,
	\end{split}
	\end{equation}
	where $\{Z_j(x)\}_{j=1}^{n}$ form a basis of $T_x^{1,0}X$, for each $x\in D$ and $\varphi(z)\in C^\infty(D,\mathbb R)$ is independent of $\theta$. We call $D$ a canonical local patch and $(z, \theta, \varphi)$ canonical coordinates centered at $x_0$.
\end{theorem}

Note that Theorem~\ref{t-gue170720y} holds if $X$ is not strongly pseudoconvex. 

On the BRT coordinate $D$, the action of the partial Cauchy-Riemann operator is the following
\begin{equation*}
\bar \partial_{b}u=\sum_{j=1}^{n}(\frac{\partial u}{\partial\bar z_{j}} - i\frac{\partial \varphi}{\partial\bar z_{j}}\frac{\partial u}{\partial\theta} )d\bar z_{j}.
\end{equation*}
We can check that 
\begin{equation*}
\omega_{0}=-d\theta+i\sum_{j=1}^{n}\frac{\partial\varphi}{\partial z_{j}}dz_{j}-i\sum_{j=1}^{n}\frac{\partial\varphi}{\partial \bar z_{j}}d\bar z_{j}.
\end{equation*}
Hence the Levi form is 
\begin{equation}
\mathcal{L}_{X}=-\frac{1}{2i}d\omega_{0}\big|_{T^{1,0}X}=\partial \bar \partial \varphi.
\end{equation}
If $u\in H^{0}_{b,m}(X)$, then $\bar \partial_{b}u=0$.
It is equivalent to that
\begin{equation*}
\frac{\partial u}{\partial\bar z_{j}} - i\frac{\partial \varphi}{\partial\bar z_{j}}\frac{\partial u}{\partial\theta}=0,\ \forall j.
\end{equation*}
Moreover, since $Tu=imu$, $u$ can be written locally as
\begin{equation*}
u\big|_{D}=e^{im\theta}\tilde u(z).
\end{equation*}
Then 
\begin{equation}\label{e-gue170723}
\begin{split}
&\frac{\partial \tilde u}{\partial\bar z_{j}} +m\frac{\partial \varphi}{\partial\bar z_{j}}\tilde u \\
&=\frac{\partial}{\partial\bar z_{j}}(\tilde ue^{m\varphi})=
0,\ \forall j.\\
\end{split}
\end{equation}
That is to say, $\tilde ue^{m\varphi}$ is holomorphic with respect to the $(z_{1},...,z_{n})$-coordinate. 

Let $X\times \mathbb{R}$ be the complex manifold with the following holomorphic tangent bundle and complex structure $J$,
\begin{equation}\label{e-gue170723I}
\begin{split}
&T^{1,0}X\oplus\{\mathbb{C}(T-i\frac{\partial}{\partial\eta}) \},\\
&JT=\frac{\partial}{\partial\eta}, \ Ju=iu \ \text{for}\ u\in T^{1,0}X.\\
\end{split}
\end{equation}

\begin{lemma}\label{l-gue170704w}
Let $u\in\oplus_{m\in\mathbb Z, |m|\leq N}C^{\infty}_{m}(X)$ with $\bar \partial_{b}u=0$, where $N\in\mathbb N$. Then there exists a unique function $v$, which is holomorphic in $X\times \mathbb{R}$ such that $v\big|_{\eta=0}=u$.
\end{lemma}

\begin{proof}
Let $D$ be a canonical local coordinate patch with canonical local coordinates $x=(z,\theta)$. On $D$, we write $u=\sum_{m\in\mathbb Z, \abs{m}\leq N}u_{m}(z)e^{im\theta}$. Note that in canonical local coordinates $x=(z,\theta)$, we have 
$T=\frac{\partial}{\partial\theta}$. Set
\begin{equation*}
v:=\sum_{m\in\mathbb Z, \abs{m}\leq N}u_{m}(z)e^{im(\theta+i\eta)}.
\end{equation*}
From $\ddbar_bu=0$, it is easy to check that $v$ is holomorphic on $D\times\Real$ with respect the complex structure \eqref{e-gue170723I} and $v|_{\eta=0}=u$. 
If there exists another function $\tilde v$ satisfying the same properties.
Then $\tilde v-v$ is holomorphic, $(\tilde v-v)\big|_{\eta=0}=0$.
So $\tilde v=v$. Thus, we can define $v$ as a global CR function on $X\times\Real$ and we have  $v|_{\eta=0}=u$. The proof is completed.
\end{proof} 

\section{Uniform estimate of Szeg\H{o} kernel functions}\label{s-gue170704}

In this section, we will give a uniform estimate of Szeg\H{o} kernel function on $X$. We keep the notations and assumptions in the  previous sections. 
We first recall a recent result about Szeg\H{o} kernel asymptotic expansion on CR manifolds with $S^1$ action due to Herrmann-Hsiao-Li
\cite{hhl}. 

For $x, y\in X$, let $d(x,y)$ denote the Riemannian distance between $x$ and $y$ induced by $\langle\,\cdot\,|\,\cdot\,\rangle$. Let $A$ be  a closed subset of $X$. Put 
$d(x,A):=\inf\set{d(x,y);\, y\in A}$.  

\begin{theorem}\label{t-gue170704}
Recall that we work with the assumptions that $X$ is a compact connected strongly pseudoconvex CR manifold of dimension $2n+1$, $n\geq 1$, with a transversal CR $S^1$ action. With the above notations for $X_{p_{r}}, 0\leq r\leq t-1$, there are $b_{j}(x)\in C^{\infty}(X)$, $j=0,1,2,\ldots$, such that for any $r=0,1,\ldots,t-1$, any differential operator $P_{\ell}:C^{\infty}(X)\To C^{\infty}(X)$ of order $\ell\in\mathbb{N}_{0}$ and every $N\in \mathbb{N}$, there are $\varepsilon_{0}>0$ and $C_{N}$ independent of $m$ with the following estimate
\begin{equation*}
\begin{split}
&\Bigl|P_{\ell}\Bigl(S_{m} (x) -\sum_{s=1} ^{p_{r}}e^{\frac{2\pi(s-1)}{p_{r}}mi} \sum_{j=0}^{N-1}m^{n-j}b_{j}(x)
\Bigr)\Bigr|\\
&\leq C_{N}
\Bigl(m^{n-N}+m^{n+\frac{\ell}{2}} e^{-m\varepsilon_{0}d(x,X_{sing}^{r+1})^{2} }    \Bigr),\ \forall m\geq 1,\ \forall x\in X_{p_{r}}, \\
\end{split}
\end{equation*}
where $b_{0}(x)\geq \epsilon>0$ on $X$ for some universal constant $\epsilon$.
\end{theorem}

Note that when $m$ is a multiple of $p_{r}$, then $\sum_{s=1} ^{p_{r}}e^{\frac{2\pi(s-1)}{p_{r}}mi}$ is equal to $p_{r}$. When
$m$ is not a multiple of $p_{r}$, then $\sum_{s=1} ^{p_{r}}e^{\frac{2\pi(s-1)}{p_{r}}mi}$ is equal to $0$.

\begin{corollary}\label{c-gue170722a}
With the above notations and assumptions, we have
\begin{equation*}
S_{m}(x)\leq Cm^{n}, \ \forall m\geq 1, \ x\in X,
\end{equation*}
where $C>0$ is a constant independent of $m$. 

Fix $r=0,1,\ldots,t-1$. There is a $m_0>0$ such that for every $m\geq m_0$, $p_r|m$, we have 
\begin{equation*}
S_{m}(x)\geq m^{n}(p_rb_{0}(x)-c_{1}e^{-m\varepsilon_{0}d(x,X_{{\rm sing\,}}^{r+1})^{2}}-c_1\frac{1}{m})
\end{equation*}
for any $x\in X_{p_{r}}$, where $c_1>0$ is a constant independent of $m$. 
\end{corollary}

\begin{corollary}\label{c-gue170722Hyc}
With the above notations and assumptions, let $r=0$, we have
\begin{equation*}
\lim_{m\To\infty}\frac{S_{m}(x)}{m^{n}}=b_{0}(x), \ \forall x\in X_{{\rm reg\,}}.
\end{equation*}
\end{corollary}

Let $x, x_1\in X$. We have 
\begin{equation}\label{e-gue170721cw}
\begin{split}
&S_m(x)=S_m(x_1)+R_m(x,x_1),\\
&R_m(x,x_1)=\int^1_0\frac{\pr}{\pr t}\Bigr(S_m(tx+(1-t)x_1)\Bigr)dt.
\end{split}
\end{equation}
By Theorem~\ref{t-gue170704} with $l=1$, 
we have the following

\begin{corollary}\label{c-gue170721cw}
We have
\[\abs{R_m(x,x_1)}\leq c_2m^{n+\frac{1}{2}}d(x,x_1),\ \ \forall (x,x_1)\in X\times X,\]
where $c_{2}>0$ is a constant independent of $m$.
\end{corollary}

The main result in this section is the following
\begin{theorem}\label{t-gue170704r}
There exist positive integers $k_{1}<\cdot\cdot\cdot<k_{t-1}$ independent of $m$ and $m_0>0$, such that for all $m\geq m_0$ with $p_{j}|m$, $j=0,1,\ldots,t-1$, we have
\begin{equation*}
\frac{1}{C}m^n\leq S_{m}(x)+S_{k_{1}m}(x)+\cdot\cdot\cdot+S_{k_{t-1}m}(x)\leq Cm^{n}, \\ \forall x\in X,
\end{equation*}
where $S_{k_{j}m}(x)$ is the Szeg\H{o} kernel function associated to $H^0_{b,k_{j}m}(X)$ and $C>1$ is a constant independent of $m$. 
\end{theorem}

\begin{proof}
Put $X^0_{{\rm sing\,}}:=X_{{\rm reg\,}}$. We claim that for every $j\in\set{0,1,\ldots,t-1}$, we can find $k_0:=1<k_1<\cdots<k_{t-1-j}$ and $m_0>0$ such that for all $m\geq m_0$ with $p_s|m$, $s=j, j+1,\ldots,t-1$, we have 
\begin{equation}\label{e-gue170722y}
\frac{1}{C}m^n\leq S_{m}(x)+S_{k_{1}m}(x)+\cdot\cdot\cdot+S_{k_{t-1-j}m}(x)\leq Cm^{n},\ \ \forall x\in X^j_{{\rm sing\,}},
\end{equation}
where $C>1$ is a constant independent of $m$. 

We prove the claim \eqref{e-gue170722y} by induction over $j$. Let $j=t-1$. Since $X^t_{{\rm sing\,}}=\emptyset$, by Theorem~\ref{t-gue170704}, we see that for all $m\gg1$ with $p_{t-1}|m$, we have 
\[S_m(x)\approx m^n\ \ \mbox{on $X^{t-1}_{{\rm sing\,}}$}.\]
The claim \eqref{e-gue170722y} holds for $j=t-1$. Assume that the claim \eqref{e-gue170722y} holds for some $0<j_0\leq t-1$. We are going to prove the claim \eqref{e-gue170722y} holds for $j_0-1$. By induction assumption, there exist positive integers $k_0:=1<k_{1}<\cdot\cdot\cdot<k_{t-1-j_0}$ independent of $m$ and $m_0>0$ such that for all $m\geq m_0$ with $p_s|m$, $s=j_0,j_0+1,\ldots,t-1$, we have
\begin{equation}\label{e-gue170722yI}
\frac{1}{C}m^n\leq A_m(x):=S_{m}(x)+S_{k_{1}m}(x)+\cdot\cdot\cdot+S_{k_{t-1-j_0}m}(x)\leq Cm^{n}, \ \ \forall x\in X^{j_0}_{{\rm sing\,}},
\end{equation}
where $C>1$ is a constant independent of $m$. In view of Corollary~\ref{c-gue170722a}, we see that there is a large constant $C_0>1$ and $m_1>0$ such that for all
$m\geq m_1$ with $p_{j_0-1}|m$ and all $x\in X_{p_{j_0-1}}$ with $d(x, X^{j_0}_{{\rm sing\,}})\geq\frac{C_0}{\sqrt{m}}$, we have 
\begin{equation}\label{e-gue170722yc}
S_m(x)\geq cm^n,
\end{equation}
where $c>0$ is a constant independent of $m$. Fix $C_0>0$, where $C_0$ is as in the discussion before \eqref{e-gue170722yc} and let $k\in\mathbb N$ and $m\gg1$ with $p_s|m$, $s=j_0,j_0+1,\ldots,t-1$. Consider the set 
\[S_{k,m}:=\set{x\in X_{p_{j_0-1}};\, d(x, X^{j_0}_{{\rm sing\,}})\leq\frac{C_0}{\sqrt{km}}}.\]
Let $x\in S_{k,m}$. 
Since $X^{j_0}_{{\rm sing\,}}$ is a closed subset of $X$ by Proposition 2.2, there is a point $x_2\in X^{j_0}_{{\rm sing\,}}$ such that $d(x,x_{2})=d(x,X^{j_0}_{{\rm sing\,}})$.
By \eqref{e-gue170721cw}, we write
\begin{equation}\label{e-gue170722pI}
\begin{split}
&A_m(x)=S_{m}(x)+S_{k_{1}m}(x)+\cdot\cdot\cdot+S_{k_{t-1-j_0}m}(x)\\
&=\Bigr(S_m(x_2)+S_{k_{1}m}(x_2)+\cdot\cdot\cdot+S_{k_{t-1-j_0}m}(x_2)\Bigr)\\
&\quad\quad+\Bigr(R_m(x,x_2)+R_{k_{1}m}(x,x_2)+\cdot\cdot\cdot+R_{k_{t-1-j_0}m}(x,x_2)\Bigr)\\
&=A_{m}(x_{2})(1+v_{m}(x,x_2)),
\end{split}
\end{equation}
where 
\[v_m(x,x_2):=(A_m(x_2))^{-1}\Bigr(R_m(x,x_2)+R_{k_{1}m}(x,x_2)+\cdot\cdot\cdot+R_{k_{t-1-j_0}m}(x,x_2)\Bigr).\]
Then with Corollary~\ref{c-gue170721cw}, 
\begin{equation}\label{e-gue170722p}
|v_{m}|\lesssim \frac{C_0}{\sqrt{km}}m^{-n}m^{n+\frac{1}{2}}\lesssim \frac{C_0}{\sqrt{k}}.
\end{equation}
From \eqref{e-gue170722pI} and \eqref{e-gue170722p}, we see that there is a large constant $k_{t-j_0}$ and $m_2>0$ such that for all $m\geq m_2$ with $p_s|m$, $s=j_0,j_0+1,\ldots,t-1$, we have 
\begin{equation}\label{e-gue170722pII}
A_m(x)\geq\hat c m^n,\ \ \forall x\in S_{k_{t-j_0}m}:=\set{x\in X_{p_{j_0-1}};\, d(x, X^{j_0}_{{\rm sing\,}})\leq\frac{C_0}{\sqrt{k_{t-j_0}m}}},
\end{equation}
where $\hat c>0$ is a constant independent of $m$. In view of \eqref{e-gue170722yc}, we see that  for all $m\geq\max\set{m_1, m_2}$ with $p_{j_0-1}|m$, we have 
\begin{equation}\label{e-gue170722yca}
S_{k_{t-j_0}m}(x)\geq \Td cm^n,\ \ \mbox{$\forall x\in X_{p_{j_0-1}}$ with $d(x, X^{j_0}_{{\rm sing\,}})\geq\frac{C_0}{\sqrt{k_{t-j_0}m}}$},
\end{equation}
where $\Td c>0$ is a constant independent of $m$. From \eqref{e-gue170722yca} and \eqref{e-gue170722pII}, we get the  claim \eqref{e-gue170722y} for $j=j_0-1$. By induction assumption, we get the claim \eqref{e-gue170722y} and the theorem follows then. 
\end{proof}

\section{Equidistribution on CR manifolds}

This section is devoted to proving Theorem~\ref{t-gue170704ryz}.  For simplicity, we assume that $X=X_{p_0}\bigcup X_{p_1}$, $p_0=1$. The proof of general case is similar. 
Let $k_1$ be as in Theorem~\ref{t-gue170704r}. 
Let $\alpha=[1,p_1]=p_1$. 
We recall some notations used in Section~\ref{s-gue170727}. For each $m\in\mathbb N$, put $A_m(X):=H^0_{b,\alpha m}(X)\bigcup H^0_{b,\alpha k_1m}(X)$, $SA_m(X):=\set{g\in A_m(X);\, (\,g\,|\,g\,)=1}$ and let 
$d\mu_m$ to denote the normalized Haar measure on the unit sphere $SA_m(X)$.  We consider the probability space $\Omega(X):=\prod^\infty_{m=1}SA_m(X)$ with the probability measure $d\mu:=\prod^\infty_{m=1}d\mu_m$. 

We first recall briefly the Lelong-Poincar\'{e} formula (see \cite[III-2.15]{dj} and \cite[Theorem 2.3.3]{mm}). 
\begin{proposition}\label{p-gue1804}
Let $Y$ be a complex manifold and $h$ be a meromorphic function on $Y$, which does not vanish identically on any connected component of $Y$. Then $h$ is locally integrable on $Y$ and satisfies the following
\begin{equation}\label{e-gue180404}
\langle[h=0],w\rangle=\int_{\{h=0\}}w=\frac{i}{\pi}\int\partial\ddbar\log|h|\wedge w,
\end{equation}
where $w$ is any test form on $Y$.
\end{proposition}

Let $u\in S A_m(X)$ and let $v(z,\theta,\eta)$ be holomorphic function on $X\times\Real$ with $v|_{\eta=0}=u$. 
For simplicity, let $m_1:=\alpha m$, $m_2:=\alpha k_1m$. On $D$, we write 
\begin{equation*}
u=u_{1}+u_{2}=\tilde u_{1}(z)e^{im_{1}\theta}+\tilde u_{2}(z)e^{im_{2}\theta}\in H^{0}_{b,m_{1}}(X)\oplus H^{0}_{b,m_{2}}(X).
\end{equation*}
Then, 
\begin{equation*}
v=\tilde u_{1}(z)e^{im_{1}\theta-m_{1}\eta}+\tilde u_{2}(z)e^{im_{2}\theta-m_{2}\eta}.
\end{equation*}

Let $g\in\Omega_{0}^{2n}(X\times\Real)$, $\langle[v=0], g\rangle$ is defined in \eqref{e-reI}.
Denote by $\tilde\partial$ (resp. $\bar{\tilde\partial}$) the $\partial$-operator (resp. $\ddbar$-operator)
with respect to the complex structure in \eqref{e-gue180419}. Since $v$ is holomorphic with the complex structure, then by the Lelong-Poincar\'{e} formula, we have
\begin{equation}\label{e-gue1804191}
\langle[v=0], g\rangle=\frac{i}{2\pi}\int\tilde\partial\bar{\tilde\partial}\log|v|^{2}\wedge g,
\end{equation}
This is globally defined which is independent of the choice of BRT coordinates in the following local calculation. 

Let $D$ be a local BRT canonical coordinate patch with canonical local coordinates $(z,\theta,\varphi)$. Let $x=(x_1,\ldots,x_{2n+1})=(z,\theta)$, $z_j=x_{2j-1}+ix_{2j}$, $j=1,\ldots,n$. 
We choose a partition of unity $\{\psi_{\ell}\}$ on $X$,
and consider $\psi_{\ell}g$ with $\text{supp}\psi_{\ell}\subset D$.
So we can assume supp$g\subset D\times\Real$. On $D\times\Real$, we have
\begin{equation}\label{e-gue1804192}
\begin{split}
&\tilde\partial=\sum_{j=1}^{n} (\frac{\partial}{\partial z_j}+i\frac{\partial\varphi(z)}{\partial z_j}\frac{\partial}{\partial\theta} ) dz_{j}+\frac{1}{2}(\frac{\partial}{\partial\theta}-i\frac{\partial}{\partial\eta})(-\omega_{0}+id\eta) \\
&\bar{\tilde\partial}=\sum_{j=1}^{n} (\frac{\partial}{\partial \bar z_j}-
i\frac{\partial\varphi(z)}{\partial\bar z_j}\frac{\partial}{\partial\theta} ) d\bar z_{j}+\frac{1}{2}(\frac{\partial}{\partial\theta}+i\frac{\partial}{\partial\eta})(-\omega_{0}-id\eta).
\end{split}
\end{equation}
Recall that
\begin{equation*}
\omega_{0}=-d\theta+i\sum_{j=1}^{n}\frac{\partial\varphi}{\partial z_{j}}dz_{j}-i\sum_{j=1}^{n}\frac{\partial\varphi}{\partial \bar z_{j}}d\bar z_{j}.
\end{equation*}
Then 
\begin{equation*}
\begin{split}
-\omega_{0}+id(\eta-\varphi)&=d\theta-i\sum_{j=1}^{n}\frac{\partial\varphi}{\partial z_{j}}dz_{j}+i\sum_{j=1}^{n}\frac{\partial\varphi}{\partial \bar z_{j}}d\bar z_{j}+id\eta
-i\sum_{j=1}^{n}\frac{\partial\varphi}{\partial z_{j}}dz_{j}-i\sum_{j=1}^{n}\frac{\partial\varphi}{\partial
 \bar z_{j}}d\bar z_{j}\\
&=d\theta-2i\sum_{j=1}^{n}\frac{\partial\varphi}{\partial z_{j}}dz_{j}+id\eta.
\end{split}
\end{equation*}
Denote by $\partial$ and $\ddbar$ the standard $\pr$-operator and $\ddbar$-operator on $(z,\theta+i\eta)$-coordinates. For simplicity, let $h$ be a function (or form) on $X\times\Real$,
we have
\begin{equation}\label{e-gue1804193}
\begin{split}
(\tilde\partial h)(z,\theta,\eta-\varphi)&
=\sum_{j=1}^{n} (\frac{\partial h}{\partial z_j}(z,\theta,\eta-\varphi)
+i\frac{\partial\varphi(z)}{\partial z_j}\frac{\partial h}{\partial\theta} (z,\theta,\eta-\varphi)) dz_{j}\\
&+\frac{1}{2}(\frac{\partial h}{\partial\theta}(z,\theta,\eta-\varphi)
-i\frac{\partial h}{\partial\eta}(z,\theta,\eta-\varphi))(-\omega_{0}+id(\eta-\varphi))\\
&=\sum_{j=1}^{n} (\frac{\partial h}{\partial z_j}(z,\theta,\eta-\varphi)
+i\frac{\partial\varphi(z)}{\partial z_j}\frac{\partial h}{\partial\theta} (z,\theta,\eta-\varphi)) dz_{j}\\
&+\frac{1}{2}(\frac{\partial h}{\partial\theta}(z,\theta,\eta-\varphi)
-i\frac{\partial h}{\partial\eta}(z,\theta,\eta-\varphi))(d\theta-
2i\sum_{j=1}^{n}\frac{\partial\varphi}{\partial z_{j}}dz_{j}+id\eta)\\
&=\sum_{j=1}^{n} (\frac{\partial h}{\partial z_j}(z,\theta,\eta-\varphi)
-\frac{\partial\varphi(z)}{\partial z_j}\frac{\partial h}{\partial\eta} (z,\theta,\eta-\varphi)) dz_{j}\\
&++\frac{1}{2}(\frac{\partial h}{\partial\theta}(z,\theta,\eta-\varphi)
-i\frac{\partial h}{\partial\eta}(z,\theta,\eta-\varphi))(d\theta+id\eta)\\
&=\partial(h(z,\theta,\eta-\varphi)).
\end{split}
\end{equation}
Similarly we have 
\begin{equation}\label{e-gue1804194}
(\bar{\tilde\partial} h)(z,\theta,\eta-\varphi)=\ddbar(h(z,\theta,\eta-\varphi)), \ \
(\tilde\partial\bar{\tilde\partial} h)(z,\theta,\eta-\varphi)=\partial\ddbar(h(z,\theta,\eta-\varphi)).
\end{equation}

Fix $\chi(\eta)\in C^\infty_0(\Real)$ with $\int\chi(\eta)d\eta=1$. 
Let $f\in\Omega_{0}^{n-1,n-1}(D)$.
Note that $\frac{\pr}{\pr\ol z_j}v(z,\theta,\eta-\varphi(z))=0$, $j=1,\ldots,n$, $(\frac{\pr}{\pr\theta}+i\frac{\pr}{\pr\eta})v(z,\theta,\eta-\varphi(z))=0$. 
From this observation, \eqref{e-gue1804191}, \eqref{e-gue1804193},
\eqref{e-gue1804194} and the Lelong-Poincar\'{e} formula, 
we have
\begin{equation}\label{e-gue1804183}
\begin{split}
&\langle[v(z,\theta,\eta)=0], f(z,\theta)\wedge\omega_0(z,\theta)\wedge\chi(\eta)d\eta\rangle \\
&=\frac{i}{2\pi}\int\partial\ddbar\log|v(z,\theta,\eta-\varphi)|^{2}\wedge
f(z,\theta)\wedge\omega_0(z,\theta)\wedge\chi(\eta-\varphi)d(\eta-\varphi).
\end{split}
\end{equation}


To prove Theorem \ref{t-gue170704ryz}, we only need to show that for $d\mu$-almost every $\set{u_m}\in\Omega(X)$, we have 
\begin{equation}\label{e-gue170729}
\lim_{m\To\infty}\frac{1}{m}\langle\,[v_{m}=0], f\wedge\omega_0\wedge \frac{1}{\varepsilon_m}\chi(\frac{\eta}{\varepsilon_m})d\eta\,\rangle
=\alpha\frac{1+k_{1}^{n+1}}{1+k_{1}^{n}}\frac{i}{\pi}\int_{X}\mathcal{L}_{X}\wedge f\wedge \omega_{0},
\end{equation}
where $v_m(x,\eta)\in C^\infty(X\times\Real)$ is the unique holomorphic function on $X\times\Real$ with $v_m(x,\eta)|_{\eta=0}=u_m(x)$. 

It follows from \eqref{e-gue1804183} that
\begin{equation}\label{e-gue170729y}
\begin{split}
&\langle[v=0], f\wedge\omega_0\wedge \frac{1}{\varepsilon_m}\chi(\frac{\eta}{\varepsilon_m})d\eta\rangle\\
&=\frac{i}{2\pi}\int\partial \bar \partial
\log|\tilde u_{1}(z)e^{im_{1}\theta+m_{1}(\varphi-\eta)}+\tilde u_{2}(z)e^{im_{2}\theta+m_{2}(\varphi-\eta)}|^2\\
&\quad\wedge
f(z,\theta)\wedge\omega_{0}\wedge\chi(\frac{\eta-\varphi}{\varepsilon_m})
\frac{1}{\varepsilon_m}(d\eta-d\varphi).
\end{split}
\end{equation}

Let $S_{m_1}$ (resp. $S_{m_2}$) be the Szeg\H{o} kernel functions of
$H^{0}_{b,m_1}(X)$ (resp. $H^{0}_{b,m_2}(X)$). By using 
the same arguments in Shiffman-Zelditch~\cite[Section 3]{sz} and Ma-Marinescu~\cite[Section 5.3]{mm}
and \eqref{e-gue170729y}, we deduce that for $d\mu$-almost every $\{u_{m}\}\in\Omega(X)$, we have 
\begin{equation}\label{e-gue170730}
\begin{split}
\lim_{m\To\infty}&\Bigl(\frac{1}{m}\langle\,[v_{m}=0], f\wedge\omega_0\wedge \frac{1}{\varepsilon_m}\chi(\frac{\eta}{\varepsilon_m})d\eta\,\rangle
-\frac{i}{2m\pi}\\
&\int\partial\bar \partial\log(e^{2m_{1}(\varphi-\eta)}S_{m_1}+e^{2m_{2}(\varphi-\eta)}S_{m_2})\wedge f\wedge\omega_0\wedge\chi(\frac{\eta-\varphi}{\varepsilon_m})\frac{1}{\varepsilon_m}(d\eta-d\varphi)\Bigr)=0.
\end{split}
\end{equation}
Let $F_{m}=e^{2m_{1}(\varphi-\eta)}S_{m_1}+e^{2m_{2}(\varphi-\eta)}S_{m_2}$. 

\begin{proof}[\bf Proof of Theorem~\ref{t-gue170704ryz}]
In view of \eqref{e-gue170730}, to prove Theorem~\ref{t-gue170704ryz}, it suffices to compute
\begin{equation}\label{e-gue170730yc}
\lim_{m\To\infty}\frac{i}{2m\pi}\int\partial\bar \partial \log F_{m}\wedge
f\wedge\omega_0\wedge\chi(\frac{\eta-\varphi}{\varepsilon_m})\frac{1}{\varepsilon_m}(d\eta-d\varphi).
\end{equation}
Recall that $S_{m_1}+S_{m_2}\approx m^{n}$ on $X$ (see Theorem~\ref{t-gue170704r}). 
We write $F=F_{m}, a_{1}=S_{m_1}, a_{2}=S_{m_2}$ for short. We have
\begin{equation}\label{e-gue170730Iq}
\partial\bar \partial \log F=\frac{\partial\bar \partial F}{F}-\frac{\partial F\wedge \bar \partial F}{F^{2}}.
\end{equation}
We can check that 
\begin{equation}\label{e-gue170730I}
\begin{split}
\partial F&=\partial(e^{2m_{1}(\varphi-\eta)}a_{1}+e^{2m_{2}(\varphi-\eta)}a_{2})\\
&=e^{2m_{1}(\varphi-\eta)}\partial a_{1}+e^{2m_{2}(\varphi-\eta)}\partial a_{2}\\
&\quad +2m_{1}a_{1}e^{2m_{1}(\varphi-\eta)}\partial (\varphi-\eta)+
2m_{2}a_{2}e^{2m_{2}(\varphi-\eta)}\partial (\varphi-\eta),
\end{split}
\end{equation}
\begin{equation}\label{e-gue170730II}
\begin{split}
\bar\partial F&=\bar\partial(e^{2m_{1}(\varphi-\eta)}a_{1}+e^{2m_{2}(\varphi-\eta)}a_{2})\\
&=e^{2m_{1}(\varphi-\eta)}\bar\partial a_{1}+e^{2m_{2}(\varphi-\eta)}\bar\partial a_{2}\\
&\quad +2m_{1}a_{1}e^{2m_{1}(\varphi-\eta)}\bar\partial (\varphi-\eta)+
2m_{2}a_{2}e^{2m_{2}(\varphi-\eta)}\bar\partial (\varphi-\eta).\\
\end{split}
\end{equation}
and 
\begin{equation}\label{e-gue170730b}
\begin{split}
\partial F\wedge\bar\partial F&=(e^{2m_{1}(\varphi-\eta)}\partial a_{1}+e^{2m_{2}(\varphi-\eta)}\partial a_{2}\\
&\quad +2m_{1}a_{1}e^{2m_{1}(\varphi-\eta)}\partial (\varphi-\eta)+
2m_{2}a_{2}e^{2m_{2}(\varphi-\eta)}\partial (\varphi-\eta))\\
&\quad\wedge (e^{2m_{1}(\varphi-\eta)}\bar\partial a_{1}+e^{2m_{2}(\varphi-\eta)}\bar\partial a_{2}\\
&\quad +2m_{1}a_{1}e^{2m_{1}(\varphi-\eta)}\bar\partial (\varphi-\eta)+
2m_{2}a_{2}e^{2m_{2}(\varphi-\eta)}\bar\partial (\varphi-\eta)).
\end{split}
\end{equation}

Moreover, we have 
\begin{equation}\label{e-gue170730III}
\begin{split}
\partial\bar\partial F&=\partial(e^{2m_{1}(\varphi-\eta)}\bar\partial a_{1}+e^{2m_{2}(\varphi-\eta)}\bar\partial a_{2}\\
&\quad +2m_{1}a_{1}e^{2m_{1}(\varphi-\eta)}\bar\partial (\varphi-\eta)+
2m_{2}a_{2}e^{2m_{2}(\varphi-\eta)}\bar\partial (\varphi-\eta))\\
&=e^{2m_{1}(\varphi-\eta)}\partial\bar\partial a_{1}+2m_{1}e^{2m_{1}(\varphi-\eta)}\partial(\varphi-\eta)\wedge\bar{\partial a_{1}}\\
&\quad +e^{2m_{2}(\varphi-\eta)}\partial\bar\partial a_{2}+2m_{2}e^{2m_{2}(\varphi-\eta)}\partial(\varphi-\eta)\wedge\bar{\partial a_{2}}\\
&\quad +2m_{1}a_{1}e^{2m_{1}(\varphi-\eta)}\partial\bar\partial(\varphi-\eta)
+2m_{1}\partial(a_{1}e^{2m_{1}(\varphi-\eta)})\wedge\bar\partial(\varphi-\eta)\\
&\quad +2m_{2}a_{2}e^{2m_{2}(\varphi-\eta)}\partial\bar\partial(\varphi-\eta)
+2m_{2}\partial(a_{2}e^{2m_{2}(\varphi-\eta)})\wedge\bar\partial(\varphi-\eta),
\end{split}
\end{equation}
and furthermore, we have 
\begin{equation}\label{e-gue170730a}
\begin{split}
&2m_{1}\partial(a_{1}e^{2m_{1}(\varphi-\eta)})\wedge\bar\partial(\varphi-\eta)\\
&=2m_{1}(e^{2m_{1}(\varphi-\eta)}\partial a_{1}\wedge\bar\partial(\varphi-\eta)+2m_{1}a_{1} e^{2m_{1}(\varphi-\eta)}\partial(\varphi-\eta)\wedge\bar\partial(\varphi-\eta)  )
\end{split}
\end{equation}
and 
\begin{equation}\label{e-gue170730aI}
\begin{split}
&2m_{2}\partial(a_{2}e^{2m_{2}(\varphi-\eta)})\wedge\bar\partial(\varphi-\eta)\\
&=2m_{2}(e^{2m_{2}(\varphi-\eta)}\partial a_{2}\wedge\bar\partial(\varphi-\eta)+2m_{2}a_{2} e^{2m_{2}(\varphi-\eta)}\partial(\varphi-\eta)\wedge\bar\partial(\varphi-\eta)  ).
\end{split}
\end{equation}

We first compute the following kinds of terms in \eqref{e-gue170730yc}:
\begin{equation}\label{e-gue170730ycI}
\int e^{2m_{j}(\varphi-\eta)}\partial (\varphi-\eta)\wedge\bar\partial a_{j}/F\wedge f\wedge\omega_0\wedge\chi(\frac{\eta-\varphi}{\varepsilon_m})\frac{1}{\varepsilon_m}(d\eta-d\varphi),\ \ j\in\set{1,2}.
\end{equation}
\begin{equation}\label{e-gue170730ycIb}
\int e^{2m_{j}(\varphi-\eta)}\ddbar(\varphi-\eta)\wedge\pr a_{j}/F\wedge f\wedge\omega_0\wedge\chi(\frac{\eta-\varphi}{\varepsilon_m})\frac{1}{\varepsilon_m}(d\eta-d\varphi),\ \ j\in\set{1,2}.
\end{equation}
\begin{equation}\label{e-gue170730ycII}
\int \frac{1}{m}e^{2m_{j}(\varphi-\eta)}\partial \bar\partial a_{j}/F\wedge f\wedge\omega_0\wedge\chi(\frac{\eta-\varphi}{\varepsilon_m})\frac{1}{\varepsilon_m}(d\eta-d\varphi),\ \ j\in\set{1,2}.
\end{equation}
\begin{equation}\label{e-gue170730ycIII}
\int a_{j}e^{2m_{j}(\varphi-\eta)}e^{2m_{k}(\varphi-\eta)}\partial a_{k}\wedge\bar\partial (\varphi-\eta)/F^{2}\wedge f\wedge\omega_0\wedge\chi(\frac{\eta-\varphi}{\varepsilon_m})\frac{1}{\varepsilon_m}(d\eta-d\varphi),\ \ j,k\in\set{1,2}.
\end{equation}
\begin{equation}\label{e-gue170730ycIIIb}
\int a_{j}e^{2m_{j}(\varphi-\eta)}e^{2m_{k}(\varphi-\eta)}\ddbar a_{k}\wedge\pr(\varphi-\eta)/F^{2}\wedge f\wedge\omega_0\wedge\chi(\frac{\eta-\varphi_m}{\varepsilon_m})\frac{1}{\varepsilon_m}(d\eta-d\varphi),\ \ j,k\in\set{1,2}.
\end{equation}
\begin{equation}\label{e-gue170730ych}
\int \frac{1}{m}e^{2m_{j}(\varphi-\eta)}e^{2m_{k}(\varphi-\eta)}\partial a_{j}\wedge\bar\partial a_{k}/F^{2}\wedge f\wedge\omega_0\wedge\chi(\frac{\eta-\varphi}{\varepsilon_m})\frac{1}{\varepsilon_m}(d\eta-d\varphi),\ \ j,k\in\set{1,2}.
\end{equation}
It is straightforward to check that 
\[\partial (\varphi-\eta)\wedge\omega_0\wedge (d\eta-d\varphi)=0,\ \ \ddbar(\varphi-\eta)\wedge\omega_0\wedge(d\eta-d\varphi)=0.\]
From this observation, we see that  terms \eqref{e-gue170730ycI}, \eqref{e-gue170730ycIb},\eqref{e-gue170730ycIII} and \eqref{e-gue170730ycIIIb} are zero. 

For \eqref{e-gue170730ycII} and \eqref{e-gue170730ych}, 
note that $\lim_{m\to\infty}m\varepsilon_{m}=0$, then 
$\lim_{m\to\infty}e^{2m_{j}(\varphi-\eta)}=1$ in the support of $\chi(\frac{\eta-\varphi}{\varepsilon_m})$.
From Theorem~\ref{t-gue170704} and Lebesgue dominate theorem, we have 
\begin{equation}\label{e-gue1804181}
\begin{split}
&\abs{\int \frac{1}{m}e^{2m_{j}(\varphi-\eta)}\partial \bar\partial a_{j}/F\wedge f\wedge\omega_0\wedge\chi(\frac{\eta-\varphi}{\varepsilon_m})
\frac{1}{\varepsilon_m}(d\eta-d\varphi)}\\
&\lesssim \frac{1}{m}\int_X\frac{m_{1}^{n}+m_{1}^{n+1}e^{-m_{1}\varepsilon_{0}d^{2}(x,X_{{\rm sing\,}})}}{m^n}\To 0\ \ \mbox{as $m\To\infty$},\ \ \forall j\in\set{1,2},
\end{split}
\end{equation}
and 
\begin{equation}\label{e-gue1804182}
\begin{split}
&\abs{\int \frac{1}{m}e^{2m_{j}(\varphi-\eta)}e^{2m_{k}(\varphi-\eta)}\partial a_{j}\wedge\bar\partial a_{k}/F^{2}\wedge f\wedge\omega_0\wedge\chi(\frac{\eta-\varphi}{\varepsilon_m})
\frac{1}{\varepsilon_m}(d\eta-d\varphi)}\\
&\lesssim \frac{1}{m}\int_X\frac{m_{1}^{n}+m_{1}^{n+1}e^{-m_{1}\varepsilon_{0}d^{2}(x,X_{{\rm sing\,}})}}{m^n}\To 0\ \ \mbox{as $m\To\infty$},\ \ \forall j, k\in\set{1,2}.
\end{split}
\end{equation}
From \eqref{e-gue170730b}, \eqref{e-gue170730III}, \eqref{e-gue170730a}, \eqref{e-gue170730aI} and the discussion above, 
we conclude that the only contribution terms in \eqref{e-gue170730yc} are those involving $\partial\bar \partial\varphi$, which is exactly the Levi form $\mathcal{L}_{X}$ of $X$. Then for $d\mu$-almost every $\{u_{m}\}\in\Omega(X)$, we have
\begin{equation}\label{e-gue170801}
\begin{split}
&\lim_{m\To\infty}\frac{1}{m}\langle\,[v_{m}=0], f\wedge\omega_0\wedge \frac{1}{\varepsilon_m}\chi(\frac{\eta}{\varepsilon_m})d\eta\,\rangle\\
&=\lim_{m\To\infty}\frac{i}{2m\pi}\int\partial\bar \partial \log F_{m}\wedge
f\wedge\omega_0\wedge\chi(\frac{\eta-\varphi}{\varepsilon_m})\frac{1}{\varepsilon_m}(d\eta-d\varphi)\\
&=\lim_{m\To\infty}\frac{i}{2m\pi}\int (2m_{1}a_{1}e^{2m_{1}(\varphi-\eta)}+
2m_{2}a_{2}e^{2m_{2}(\varphi-\eta)})/F_{m}\cdot\partial\bar \partial\varphi\wedge f\wedge\omega_0\wedge\chi(\frac{\eta-\varphi}{\varepsilon_{m}})
\frac{1}{\varepsilon_{m}}(d\eta-d\varphi)\\
&=\lim_{m\To\infty}\frac{i}{\pi}\int \frac{\alpha S_{\alpha m}(x)+k_1\alpha S_{\alpha k_1m}(x)}{S_{\alpha m}(x)+S_{\alpha k_1m}(x)}
\partial\bar \partial\varphi\wedge f\wedge\omega_0.
\end{split}
\end{equation}
From Corollary~\ref{c-gue170722Hyc}, Theorem~\ref{t-gue170704}, Lebesgue dominate theorem and \eqref{e-gue170801}, we deduce \eqref{e-gue170729}. Theorem~\ref{t-gue170704ryz} follows. 
\end{proof}

\section{Equidistribution on complex manifolds with strongly pseudoconvex boundary}\label{s-gue170709} 

In this section, we will prove Theorem~\ref{t-gue170703c}. 
Let $M$ be a relatively compact open subset with $C^\infty$ boundary $X$ of a complex manifold $M'$ of dimension $n+1$ with a smooth Hermitian metric $\langle\,\cdot\,|\,\cdot\,\rangle$ on its holomorphic tangent bundle $T^{1,0}M'$. From now on, we will use the same notations and assumptions as in the discussion before Theorem~\ref{t-gue170703c}. We will first recall the classical results of  Boutet de Monvel-Sj\"ostrand~\cite{BouSj76} (see also second part in~\cite{Hsiao08}). We then construct holomorphic functions with specific rate near the boundary. We first recall the H\"{o}rmander symbol spaces

\begin{definition} \label{Bd:0712101500}
Let $m\in\Real$. $S^{m}_{1, 0}(M'\times M'\times]0, \infty[)$
is the space of all $a(x, y, t)\in C^\infty(M'\times M'\times]0, \infty[)$
such that for all local coordinate patch $U$ with local coordinates $x=(x_1,\ldots,x_{2n+2})$ and all
compact sets $K\subset U$ and all $\alpha\in\mathbb N^{2n+2}_0$, $\beta\in\mathbb N^{2n+2}_0$, $\gamma\in\mathbb N_0$, there is a
constant $c>0$ such that
$\abs{\pr^\alpha_x\pr^\beta_y\pr^\gamma_t a(x, y, t)}\leq c(1+\abs{t})^{m-\abs{\gamma}}$,
$(x, y, t)\in K\times]0, \infty[$.
$S^m_{1, 0}$ is called the space of symbols of order $m$ type $(1, 0)$. We write $S^{-\infty}_{1, 0}=\bigcap S^m_{1, 0}$.

Let $S^{m}_{1, 0}(\ol M\times\ol M\times]0, \infty[)$ denote the space of
restrictions to $M\times M\times]0, \infty[$ of elements in $S^{m}_{1, 0}(M'\times M'\times]0, \infty[)$.
\end{definition}

Let $a_j\in S^{m_j}_{1, 0}(\ol M\times\ol M\times]0, \infty[)$, $j=0,1,2,\dots$, 
with $m_j\searrow -\infty$, $j\To \infty$.
Then there exists $a\in S^{m_0}_{1, 0}(\ol M\times\ol M\times]0, \infty[)$
such that $a-\sum_{0\leq j<k}a_j\in S^{m_k}_{1, 0}(\ol M\times\ol M\times]0, \infty[)$, 
for every $k\in\mathbb N$. 
If $a$ and $a_j$ have the properties above, we write $a\sim\sum^\infty_{j=0}a_j \text{ in }
S^{m_0}_{1, 0}(\ol M\times\ol M\times[0, \infty[)$.

Let $dv_M$ be the volume form on $M$ induced by $\langle\,\cdot\,|\,\cdot\,\rangle$ and let $(\,\cdot\,|\,\cdot\,)_M$ be the $L^2$ inner product on $C^\infty_0(M)$ induced by $dv_M$ and let $L^2(M)$ be the completion of $C^\infty_0(M)$ with respect to $(\,\cdot\,|\,\cdot\,)_M$. Let $H^0_{(2)}(M)=\set{u\in L^2(M);\, \ddbar u=0}$. Let 
$B: L^2(M)\To H^0(M)$ be the orthogonal projection with respect to $(\,\cdot\,|\,\cdot\,)_M$ and let $B(z,w)\in D'(M\times M)$ be the distribution kernel of $B$. 
We recall classical result of Boutet de Monvel-Sj\"strand~\cite{BouSj76}. 

\begin{theorem}\label{t-gue170715}
With the notations and assumptions above, we have
\begin{equation}\label{e-gue170716}
B(z, w)=\int^\infty_0\!\!e^{i\phi(z, w)t}b(z, w, t)dt+H(z,w),
\end{equation}
(for the precise meaning of the oscillatory integral $\int^\infty_0\!\!e^{i\phi(z, w)t}b(z, w, t)dt$, 
see Remark~\ref{Br:0712111922} below) where 
$H(z,w)\in C^\infty(\ol M\times\ol M)$, 
\begin{equation}\label{e-gue170717a}
\begin{split}
&b(z, w, t)\in S^{n+1}_{1, 0}(\ol M\times\ol M\times]0, \infty[),\\
&\mbox{$b(z, w, t)\sim\sum^\infty_{j=0}b_j(z, w)t^{n+1-j}$
in the space $S^{n+1}_{1, 0}(\ol M\times\ol M\times]0, \infty[)$},\\
&b_j(z, w)\in C^\infty(\ol M\times\ol M),\ \ j=0,1,\ldots,\\
&b_0(z, z)\neq0,\ z\in X,
\end{split}
\end{equation}
and
\begin{equation}\label{e-gue170717aI}
\begin{split}
&\phi(z, w)\in C^\infty(\ol M\times\ol M),\\
&\phi(z, z)=0,\ \ z\in X,\ \ \phi(z, w)\neq0\ \ \mbox{if}\ \ (z, w)\notin{\rm diag\,}(X\times X), \\
&{\rm Im\,}\phi(z, w)>0\ \ \mbox{if}\ \ (z, w)\notin X\times X, \\
&\mbox{$\phi(z,z)=r(z)g(z)$ on $\ol M$, $g(z)\in C^\infty(\ol M)$ with $\abs{g(z)}>c$ on $\ol M$, $c>0$ is a constant}.
\end{split}
\end{equation}
Moreover, there is a content $C>1$ such that 
\begin{equation}\label{e-gue170717aII}
\frac{1}{C}({\rm dist\,}(x,y))^2\leq \abs{d_y\phi(x,y)}^2+\abs{{\rm Im\,}\phi(x,y)}\leq C({\rm dist\,}(x,y))^2,\ \ \forall (x,y)\in X\times X,
\end{equation}
where $d_y$ denotes the exterior derivative on $X$ and ${\rm dist\,}(x,y)$ denotes the distance between $x$ and $y$ with respect to the give Hermitian metric $\langle\,\cdot\,|\,\cdot\,\rangle$ on $X$.
\end{theorem}

\begin{remark} \label{Br:0712111922}
Let $\phi$ and $b(z, w, t)$ be as in Theorem~\ref{t-gue170715}. Let
$y=(y_1,\ldots,y_{2n+1})$
be local coordinates on $X$ and extend $y_1,\ldots,y_{2n+1}$ to real smooth functions in some neighborhood of $X$.
We work with local coordinates
$w=(y_1,\ldots,y_{2n+1},r)$
defined on some neighborhood $U$ of $p\in X$.
Let $u\in C^\infty_0(U)$. Choose a cut-off function $\chi(t)\in C^\infty(\Real)$
so that $\chi(t)=1$ when $\abs{t}<1$ and $\chi(t)=0$ when $\abs{t}>2$. Set
\[(B_{\epsilon}u)(z)=\int^\infty_0\int_{\ol M}e^{i\phi(z, w)t}b(z, w, t)\chi(\epsilon t)u(w)dv_M(w)dt.\]
Since $d_y\phi\neq0$ where ${\rm Im\,}\phi=0$ (see \eqref{e-gue170717aII}),
we can integrate by parts in $y$ and $t$ and obtain
$\lim_{\epsilon\To0}(B_\epsilon u)(z)\in C^\infty(\ol M)$.
This means that
$B=\lim_{\epsilon\To0}B_\epsilon: C^\infty(\ol M)\To C^\infty(\ol M)$
is continuous.
\end{remark}

We have the following corollary of Theorem~\ref{t-gue170715}

\begin{corollary} \label{Bi-c:co1}
Under the notations and assumptions above,  we have 
\begin{equation}\label{e-gue170717f}
B(z,z)=F(z)(-r(z))^{-n-2}+G(z)\log(-r(z))\ \ \mbox{on $\ol M$}, 
\end{equation}
where $F, G\in C^\infty(\ol M)$ and $\abs{F(z)}>c$ on $X$, $c>0$ is a constant. 
\end{corollary} 

Since $C^\infty(\ol M)\bigcap H^0_{(2)}(M)$ is dense in $H^0_{(2)}(M)$ in $L^2(M)$, we can find $g_j\in C^\infty(\ol M)\bigcap H^0_{(2)}(M)$ with $(\, g_j\,|\,g_k\,)_M=\delta_{j,k}$, $j, k=1,2,\ldots$, such that the set
\begin{equation}\label{e-gue170709qm}
A(M):=\rm{span\,}\set{g_1, g_2,\ldots}
\end{equation}
is dense in $H^0_{(2)}(M)$. Moreover, for every $u\in L^2(M)$, we have 
\begin{equation}\label{e-gue170715}
\mbox{$\sum^N_{j=1}g_j(\,u\,|\,g_j\,)_M\To Bu$ in $L^2(M)$}\ \ \mbox{as $N\To\infty$}.
\end{equation}
Fix $k\in\mathbb N$, $k$ large. Fix $x_0\in M$ with $\frac{1}{2k}\leq\abs{r(x_0)}\leq\frac{1}{k}$. Let $x=(x_1,\ldots,x_{2n+2})$ be local coordinates of $M$ defined in a small neighborhood of $x_0$ with $x(x_0)=0$. Let $\chi\in C^\infty_0(\Real^{2n+2})$ with $\chi\equiv1$ near $0\in\Real^{2n+2}$. For $\varepsilon>0$, put $\chi_\varepsilon(x)=\varepsilon^{-(2n+2)}\chi(\frac{x}{\varepsilon})$. From \eqref{e-gue170715}, for every $\varepsilon>0$, $\varepsilon$ small, we have 
\begin{equation}\label{e-gue170715I}
\sum^\infty_{j=0}\abs{(\,g_j\,|\,\chi_\varepsilon\,)_M}^2=(\,B\chi_\varepsilon\,|\,\chi_\varepsilon\,)_M. 
\end{equation}
Since $B(z,w)\in C^\infty(M\times M)$, we have 
\begin{equation}\label{e-gue170715III}
\lim_{\varepsilon\To0}\Bigr(\sum^\infty_{j=1}\abs{(\,g_j\,|\,\chi_\varepsilon\,)_M}^2\Bigr)=B(x_0,x_0)m(x_0),
\end{equation}
where $m(x)dx_1\cdots dx_{2n+2}=dv_M$. 

From \eqref{e-gue170715}, for every $\varepsilon_1, \varepsilon_2>0$, $\varepsilon_1, \varepsilon_2$ small, we have 
\begin{equation}\label{e-gue170715a}
\begin{split}
&\sum^\infty_{j=0}\abs{(\,g_j\,|\,\chi_{\varepsilon_1}\,)_M-(\,g_j\,|\,\chi_{\varepsilon_2}\,)_M}^2\\
&=(\,B\chi_{\varepsilon_1}\,|\,\chi_{\varepsilon_1}\,)_M-
(\,B\chi_{\varepsilon_1}\,|\,\chi_{\varepsilon_2}\,)_M-(\,B\chi_{\varepsilon_2}\,|\,\chi_{\varepsilon_1}\,)_M+(\,B\chi_{\varepsilon_2}\,|\,\chi_{\varepsilon_2}\,)_M.
\end{split}
\end{equation}
Since $B(z,w)\in C^\infty(M\times M)$, we deduce that for every $\delta>0$, there is a $C_\delta>0$ such that for all $0<\varepsilon_1, \varepsilon_2<C_\delta$, we have 
\begin{equation}\label{e-gue170715aI}
\sum^\infty_{j=0}\abs{(\,g_j\,|\,\chi_{\varepsilon_1}\,)_M-(\,g_j\,|\,\chi_{\varepsilon_2}\,)_M}^2<\delta.
\end{equation}
Now, we can prove 

\begin{theorem}\label{t-gue170715b}
We have $\sum^\infty_{j=1}\abs{g_j(x_0)}^2=B(x_0,x_0)m(x_0)$. 
\end{theorem}

\begin{proof}
From \eqref{e-gue170715III}, it is easy to see that 
\begin{equation}\label{e-gue170715b}
\sum^\infty_{j=1}\abs{g_j(x_0)}^2\leq B(x_0,x_0)m(x_0).
\end{equation}
Let $\delta>0$ and fix $0<\varepsilon_0<C_\delta$, where $C_\delta$ is as in \eqref{e-gue170715aI}. Since $\sum^\infty_{j=1}\abs{(\,g_j\,|\,\chi_{\varepsilon_0}\,)_M}^2<\infty$, there is a $N\in\mathbb N$ such that 
\begin{equation}\label{e-gue170715bI}
\sum^{\infty}_{j=N+1}\abs{(\,g_j\,|\,\chi_{\varepsilon_0}\,)_M}^2<\delta.
\end{equation}
Now, for every $0<\varepsilon<\varepsilon_0$, from \eqref{e-gue170715aI} and \eqref{e-gue170715bI}, we have 
\begin{equation}\label{e-gue170715bII}
\begin{split}
\sum^{\infty}_{j=N+1}\abs{(\,g_j\,|\,\chi_{\varepsilon}\,)_M}^2
&\leq 2\sum^\infty_{j=N+1}\abs{(\,g_j\,|\,\chi_{\varepsilon}\,)_M-(\,g_j\,|\,\chi_{\varepsilon_0}\,)_M}^2+2\sum^\infty_{j=N+1}\abs{(\,g_j\,|\,\chi_{\varepsilon_0}\,)_M}^2\\
&\leq 2\sum^\infty_{j=1}\abs{(\,g_j\,|\,\chi_{\varepsilon}\,)_M-(\,g_j\,|\,\chi_{\varepsilon_0}\,)_M}^2+2\sum^\infty_{j=N+1}\abs{(\,g_j\,|\,\chi_{\varepsilon_0}\,)_M}^2\\
&\leq 4\delta.
\end{split}
\end{equation}
From \eqref{e-gue170715bII}, we deduce that 
\begin{equation}\label{e-gue170717y}
\limsup_{\varepsilon\To0}\sum^{\infty}_{j=N+1}\abs{(\,g_j\,|\,\chi_{\varepsilon}\,)_M}^2\leq 4\delta.
\end{equation}
Now, 
\begin{equation}\label{e-gue170717yI}
\begin{split}
&\sum^\infty_{j=1}\abs{g_j(x_0)}^2\geq\sum^N_{j=1}\abs{g_j(x_0)}^2=\lim_{\varepsilon\To0}\sum^N_{j=1}\abs{(\,g_j\,|\,\chi_\varepsilon\,)_M}^2\\
&\geq\liminf_{\varepsilon\To0}\Bigr(\sum^\infty_{j=1}\abs{(\,g_j\,|\,\chi_\varepsilon\,)_M}^2-\sum^\infty_{N+1}\abs{(\,g_j\,|\,\chi_\varepsilon\,)_M}^2\Bigr)\\
&\geq \liminf_{\varepsilon\To0}\sum^\infty_{j=1}\abs{(\,g_j\,|\,\chi_\varepsilon\,)_M}^2-\limsup_{\varepsilon\To0}\sum^\infty_{N+1}\abs{(\,g_j\,|\,\chi_\varepsilon\,)_M}^2.
\end{split}
\end{equation}
From \eqref{e-gue170715III}, \eqref{e-gue170717y} and \eqref{e-gue170717yI}, we deduce that 
\[\sum^\infty_{j=1}\abs{g_j(x_0)}^2\geq B(x_0,x_0)m(x_0)-4\delta.\]
Since $\delta$ is arbitrary, we conclude that 
\begin{equation}\label{e-gue170717yII}
\sum^\infty_{j=1}\abs{g_j(x_0)}^2\geq B(x_0,x_0)m(x_0).
\end{equation}
From \eqref{e-gue170717yII} and \eqref{e-gue170715b}, the theorem follows. 
\end{proof}

From Theorem~\ref{t-gue170715b} and \eqref{e-gue170717f}, we deduce that there is a $N_{x_0}\in\mathbb N$ such that 
\begin{equation}\label{e-gue170717ycq}
\abs{r^{n+2}(x_0)\sum^{N_{x_0}}_{j=1}\abs{g_j(x_0)}^2}\geq\frac{1}{2}\abs{F(x_0)},
\end{equation}
where $F$ is as in \eqref{e-gue170717f}. Let 
\begin{equation}\label{e-gue17717h}
h_{x_0}:=\frac{1}{\sum^{N_{x_{0}}}_{j=1}\abs{g_j(x_0)}^2}\sum^{N_{x_{0}}}_{j=1}g_j(x)\abs{\ol g_j(x_0)}.
\end{equation}
Then, $h_{x_0}\in H^0_{(2)}(M)\cap C^\infty(\ol M)$ with $(\,h_{x_0}\,|\,h_{x_0}\,)_M=1$ and there is a small neighborhood $U_{x_0}$ of $x_0$ in $M$ such that
\begin{equation}\label{e-gue170717yuI}
\abs{h_{x_0}(x)}\geq\frac{1}{4}\abs{F(x)}.
\end{equation}
Assume that $\set{x\in M, \frac{1}{2k}\leq\abs{r(x)}\leq\frac{1}{k}}\subset U_{x_0}\bigcup U_{x_1}\bigcup\cdots\bigcup U_{x_{a_k}}$ and let 
$h_{x_j}$ be as in \eqref{e-gue17717h}, $j=0,1,\ldots,a_k$. Take $\beta_k\in\mathbb N$ be a large number so that 
\[\set{h_{x_0}, h_{x_1},\ldots, h_{x_{a_k}}}\subset{\rm span\,}\set{g_1, g_2,\ldots, g_{\beta_k}}.\]
From \eqref{e-gue170717yuI}, it is easy to see that 
\begin{equation}\label{e-gue170717fhI}
\abs{r^{n+2}(x)\sum^{\beta_k}_{j=1}\abs{g_j(x)}^2}\geq\frac{1}{4}\abs{F(x)}\ \ \mbox{on $\set{x\in M, \frac{1}{2k}\leq\abs{r(x)}\leq\frac{1}{k}}$}.
\end{equation}
Note that  $\abs{F(x)}>c$ on $X$, where $c>0$ is a constant. From this observation and \eqref{e-gue170717fhI}, we get 

\begin{theorem}\label{t-gue170717}
There is a $k_0\in\mathbb N$ such that for every $k\in\mathbb N$, $k\geq k_0$, we can find $\beta_k\in\mathbb N$ such that 
\begin{equation}\label{e-gue170717hy}
\abs{r^{n+2}(x)\sum^{\beta_k}_{j=1}\abs{g_j(x)}^2}\geq c_0\ \ \mbox{on $\set{x\in M, \frac{1}{2k}\leq\abs{r(x)}\leq\frac{1}{k}}$},
\end{equation}
where $c_0>0$ is a constant independent of $k$. 
\end{theorem}

Let $b_j=\beta_{k_0+j}\in\mathbb N$, $j=1,2,\ldots$, where $\beta_j$ and $k_0$ are  as in Theorem~\ref{t-gue170717}. 
For every $m\in\mathbb N$, let $A_m(M)$, $SA_m(M)$ and $d\mu_m$ be as in the discussion before \eqref{e-gue170709I}. 
Let $\beta:=\set{b_j}^\infty_{j=1}$ and let $\Omega(M,\beta)$ and $d\mu(\beta)$ be as in \eqref{e-gue170709I} and \eqref{e-gue170709II} respectively. For each $k=1,2,3,\ldots$, let 
\[P_k(x):=\sum^{b_k}_{j=1}\abs{g_j(x)}^2.\]
Let $u_k\in SA_{b_k}(M)$. Then, $u_k$ can be written as $u_k=\sum^{b_k}_{j=1}\lambda_jg_j$ with $\sum^{b_k}_{j=1}\abs{\lambda_j}^2=1$. We have 

\begin{theorem}\label{t-ue170717j}
With the notations and assumptions above, fix $\psi\in C^\infty_0([-1,-\frac{1}{2}])$. Then, for $d\mu(\beta)$-almost every $u=\set{u_k}\in\Omega(M,\beta)$, we have 
\begin{equation}\label{e-gue170717jI}
\lim_{k\To\infty}\Bigr(\langle\,[u_k=0], (2i)kr\psi(kr)\phi\wedge\pr r\wedge\ddbar r\,\rangle+\frac{1}{\pi}\int_{\ol M}\Bigr(\log P_k(x)\Bigr)kr\psi(kr)\pr\ddbar\phi\wedge\pr r\wedge\ddbar r\Bigr)=0,
\end{equation}
for all $\phi\in C^\infty(\ol M,B^{*n-1,n-1}M')$. 
\end{theorem}

\begin{proof}
The proof essentially follows from Shifffman-Zelditch~\cite{sz}, we only sketch the proof. By using density argument, we only need to prove that for any $\phi\in C^\infty(\ol M, B^{*n-1,n-1}T^*M')$, there exist $d\mu(\beta)$-almost every $u=\set{u_k}\in\Omega(M,\beta)$, such that 
\begin{equation}\label{e-gue170717jIy}
\lim_{k\To\infty}\Bigr(\langle\,[u_k=0], (2i)kr\psi(kr)\phi\wedge\pr r\wedge\ddbar r\,\rangle+\frac{1}{\pi}\int_{\ol M}\Bigr(\log P_k(x)\Bigr)kr\psi(kr)\pr\ddbar\phi\wedge\pr r\wedge\ddbar r\Bigr)=0.
\end{equation}
We claim that 
\begin{equation}\label{e-gue170717jyI}
\begin{split}
&R_k:=\int_{S^{2b_k-1}}\Bigr|\langle\,[\sum^{b_k}_{j=1}\lambda_jg_j=0], (2i)kr\psi(kr)\phi\wedge\pr r\wedge\ddbar r\,\rangle\\
&\quad+\frac{1}{\pi}\int_{\ol M}(\log P_k(x))kr\psi(kr)\pr\ddbar\phi\wedge\pr r\wedge\ddbar r\Bigr|^2d\mu_{b_k}(\lambda)=O(\frac{1}{k^2}).
\end{split}
\end{equation}
From \eqref{e-gue170717jyI}, we see that $\sum^\infty_{k=1}R_k<+\infty$ and by Lebesgue measure theory, we get \eqref{e-gue170717jIy}. Hence, we only need to prove \eqref{e-gue170717jyI}. 

For $(x, y)\in\ol M\times\ol M$, put 
\[\begin{split}
&Q_k(x,y):=\int_{S^{2b_k-1}}\log\Bigr(\frac{\abs{\sum^{b_k}_{j=1}\lambda_jg_j(x)}^2}{P_k(x)}\Bigr)
\log\Bigr(\frac{\abs{\sum^{b_k}_{j=1}\lambda_jg_j(y)}^2}{P_k(y)}\Bigr)d\mu_{b_k}(\lambda),\\
&f_k:=-\frac{1}{\pi}kr\psi(kr)\pr\ddbar\phi(y)\wedge\pr r\wedge\ddbar r\in C^\infty_0(M,T^{*n+1,n+1}M'). 
\end{split}\]
By using the same argument in~\cite{sz} (see also Theorem 5.3.3 in~\cite{mm}), we can check that 
\begin{equation}\label{e-gue170717yc}
R_k=\int_{\ol M\times\ol M} Q_k(x,y)f_k(x)\wedge f_k(y).
\end{equation}
Moreover, from Lemma 5.3.2 in~\cite{mm}, there is a constant $C_k>0$ independent of $(x,y)\in\ol M\times \ol M$ such that 
\begin{equation}\label{e-gue170717yca}
\abs{Q_k(x,y)-C_k}\leq C,\ \ \forall (x, y)\in M\times M,
\end{equation}
where $C>0$ is a constant independent of $k$. From \eqref{e-gue170717yca}, it is easy to check that 
\begin{equation}\label{e-gue170717ycI}
\abs{\int_{\ol M\times\ol M} \Bigr(Q_k(x,y)-C_k\Bigr)f_k(x)\wedge f_k(y)}=O(\frac{1}{k^2}). 
\end{equation}
By using integration by parts, we see that $\int_{\ol M\times\ol M} \Bigr(Q_k(x,y)-C_k\Bigr)f_k(x)\wedge f_k(y)=R_k$. From this observation and \eqref{e-gue170717ycI}, 
the claim \eqref{e-gue170717jyI} follows. 
\end{proof}

\begin{proof}[\bf Proof of Theorem~\ref{t-gue170703c}]

In view of Theorem~\ref{t-ue170717j}, we only need to show that 
\[\lim_{k\To\infty}-\frac{1}{\pi}\int_{\ol M}(\log P_k(x))kr\psi(kr)\pr\ddbar\phi\wedge\pr r\wedge\ddbar r=-(n+2)\frac{i}{2\pi}c_0\int_X\mathcal{L}_X\wedge\omega_0\wedge\phi,\]
where $c_0=\int_\Real\psi(x)dx$. Now, 
\begin{equation}\label{e-gue170717yhc}
\begin{split}
&-\frac{1}{\pi}\int_{\ol M}\Bigr(\log P_k(x)\Bigr)kr\psi(kr)\pr\ddbar\phi\wedge\pr r\wedge\ddbar r\\
&=-\frac{1}{\pi}\int_{\ol M}\Bigr(\log (P_k(x)(-r)^{n+2}(x))\Bigr)kr\psi(kr)\pr\ddbar\phi\wedge\pr r\wedge\ddbar r\\
&\quad+\frac{n+2}{\pi}\int_{\ol M}\Bigr(\log(-r)(x)\Bigr)kr\psi(kr)\pr\ddbar\phi\wedge\pr r\wedge\ddbar r\\
&=-\frac{1}{\pi}\int_{\ol M}\Bigr(\log (P_k(x)(-r)^{n+2}(x))\Bigr)kr\psi(kr)\pr\ddbar\phi\wedge\pr r\wedge\ddbar r\\
&\quad-i\frac{n+2}{2\pi}\int_{\ol M}\Bigr(\log(-r)(x)\Bigr)kr\psi(kr)\pr\ddbar\phi\wedge\omega_0\wedge dr,
\end{split}
\end{equation}
where $\omega_0=J(dr)$, $J$ is the standard complex structure map on $T^*M'$. 

From Theorem~\ref{t-gue170717}, it is easy to see that 
\begin{equation}\label{e-gue170717e}
\lim_{k\To+\infty}-\frac{1}{\pi}\int_{\ol M}\Bigr(\log (P_k(x)(-r)^{n+2}(x))\Bigr)kr\psi(kr)\pr\ddbar\phi\wedge\pr r\wedge\ddbar r=0. 
\end{equation}
By using integration by parts, we have 
\begin{equation}\label{e-gue170717eI}
\begin{split}
&-i\frac{n+2}{2\pi}\int_{\ol M}\Bigr(\log(-r)(x)\Bigr)kr\psi(kr)\pr\ddbar\phi\wedge\omega_0\wedge dr\\
&=-i\frac{n+2}{2\pi}\int_{\ol M}\Bigr((\pr\ddbar \log(-r))(x)\Bigr)kr\psi(kr)\phi\wedge\omega_0\wedge dr\\
&=-i\frac{n+2}{2\pi}\int_{\ol M}\pr\ddbar r(x) k\psi(kr)\phi\wedge\omega_0\wedge dr\\
&\To-(n+2)\frac{i}{2\pi}c_0\int_X\mathcal{L}_X\wedge\omega_0\wedge\phi \ \ \mbox{as $k\To\infty$}, 
\end{split}
\end{equation}
where $c_0=\int_\Real\psi(x)dx$. From \eqref{e-gue170717yhc}, \eqref{e-gue170717e} and \eqref{e-gue170717eI}, the theorem follows then. 
\end{proof}



\bibliographystyle{plain}


\end{document}